\documentclass{amsart}

\usepackage{amsfonts}
\usepackage{graphicx}
\usepackage{xcolor}
\usepackage{amsthm}

\usepackage{cite}

\newtheorem{theorem}{Theorem}
\newtheorem{corollary}[theorem]{Corollary}
\newtheorem{lemma}[theorem]{Lemma}
\newtheorem{proposition}[theorem]{Proposition}

\theoremstyle{plain}

\newtheorem{definition}{Definition}

\newtheorem{remark}{Remark}

\numberwithin{equation}{section}

\newcommand{\R}{{\mathbb R}}
\newcommand{\N}{{\mathbb N}}

\renewcommand{\P}{{\mathbb P}}

\newcommand{\cI}{{\mathcal I}}
\newcommand{\cR}{{\mathcal R}}
\newcommand{\cQ}{{\mathcal Q}}

\newcommand{\cA}{{\mathcal A}}
\newcommand{\cW}{{\mathcal W}}

\newcommand{\blambda}{{\boldsymbol{\lambda}}}

\newcommand{\mfS}{\mathfrak{S}}

\DeclareMathOperator{\supp}{supp}

\newcommand{\ee}{{\mathbb E}}
\newcommand{\odim}{\overline{\dim}}
\newcommand{\udim}{\underline{\dim}}

\begin{document}
\title[]{The Attainable almost sure LARGE dimensions}
\author{Kathryn E. Hare}
\address{Dept. of Pure Mathematics, University of Waterloo, Waterloo, Ont.,
Canada, N2L 3G1}
\email{kehare@uwaterloo.ca}
\author{Franklin Mendivil}
\address{Department of Mathematics and Statistics, Acadia University,
Wolfville, N.S. Canada, B4P 2R6}
\email{franklin.mendivil@acadiau.ca}
\thanks{The research of K. Hare is partially supported by NSERC 2016:03719.
The research of F. Mendivil is partially supported by NSERC\ 2019:05237.}

\begin{abstract}
In this paper we study the range of possible almost sure dimensions of random measures arising from a natural model
of random Moran measures.
Specifically, we consider the Assouad-like ``large'' $\Phi$-dimensions of these measures.
These dimensions can be tuned to consider a specific range of depths in scale and so provide refined local geometric
information.
The quasi-Assouad dimension is a well-known and important example of a ``large'' $\Phi$-dimension.

We determine the range of possible almost sure $\Phi$-dimensions for random measures  generated by the model and are supported on any given random Moran set.
We do this for both the case when the probability weights depend on the scaling factors and the case when they do not.
In the later situation, we show that usually there is a ``gap'' between the dimension of the set and that of the  smallest attainable upper dimension and largest attainable lower dimension.
As a consequence of our results, we also determine the a.s. dimensions of the underlying random Moran set.

\end{abstract}

\maketitle

\section{Introduction}

Random constructions have been investigated in the fractal geometry literature for many years
(see \cite{Fa, FMT, FT,  Gr, Ham, HM, Tr0, Tr, Tr2}), with a common focus being the dimensional properties of the resulting random objects.
The random context often helps to illuminate the robustness of the dimension in illustrating how  fluctuations might change or not 
change the dimension from that of a related deterministic model.
In this paper we continue our examination of the possible dimensions of random Moran sets and measures which arise from random iterated function
systems (or 1-variable random models).
The random set and measure are constructed by an iterative Moran construction where the contractions and probability weighting factors are randomly chosen at each step.
As with most such constructions, there are almost sure values for the dimensions of the resulting random set and measure.
The almost sure Hausdorff dimension of the random set was determined in \cite{Ham} (see also \cite{Betc}) and in \cite{Tr0} it was shown that the Hausdorff and upper box dimensions coincide almost surely for these random sets.

The Assouad and Assouad-like dimensions, in contrast with the Hausdorff and box dimensions, are defined to assess the most extreme local scaling behaviour of a set or measure and are thus sensitive to small local changes.
These dimensions all appear in dual ``upper'' and ``lower'' versions, similar to the pairing of the upper and lower box dimensions.
The Assouad dimensions have seen a surge of applications in fractal geometry in recent years (see \cite{BRT23, Fraserbook, FH, FMT, FT, FY, GHM, HH, HM, HM2, HM3, KL, Ol, Tr0, Tr, Tr2}), in part because of their ability to quantify the inhomogeneity of a set.
The most general of these Assouad-like dimensions are the $\Phi$-dimensions, first introduced in \cite{FY} and then extensively studied in \cite{GHM}.
These use a function, $\Phi$, to control the ratio of two scales at which to asymptotically compare the local ``size'' of the set or measure.
By varying $\Phi$ one can get a precise quantification of the extreme local scaling behaviour of the set or measure under consideration.
The Assouad, quasi-Assouad and $\theta$-Assouad spectrum are important examples of $\Phi$-dimensions.
There are ``large'' and ``small'' $\Phi$ dimensions, roughly depending on how separated are the two comparison scales (see the formal definition in the next section); the Assouad dimension is in this sense small, while the quasi-Assouad and $\theta$-Assouad spectrum are large.

Following on the work started in \cite{FMT, Tr0, Tr}, the authors of \cite{HM, HM2, HM3} investigated the $\Phi$-dimensions of these random Moran sets and measures.
In \cite{HM, HM2} they determined formulas for the almost sure dimensions of the random measures under various different scenarios on the model.
The analysis is different for the small and large $\Phi$ functions, as well as for the upper and lower versions of the dimensions.
The same threshold phenomena of small versus large was also found in previous work \cite{GHM,Tr2}.

In this paper we examine the range of attainable large $\Phi$ dimensions of random measures on a given random set, following up on a similar examination for the attainable small $\Phi$ in \cite{HM3}.
More specifically, imagine that we are given a random iterated function system which will produce a random set $F_\blambda$.
Based on this, we now add a model for randomly choosing probability weights at each step of the construction.
This will produce a random probability measure $\mu_\blambda$ supported on $F_\blambda$.
It is known (see \cite{HH}) that the $\Phi$ dimension of $F_\blambda$ bounds that of $\mu_\blambda$, but which of the allowable
$\dim_\Phi \mu_\blambda$ are attainable?
The answer turns out to depend on the details of the underlying random process.
If the choice of probability weights can be dependent on $F_\blambda$ then the entire allowable range is achievable.
On the other hand, if the probability weights are independent of $F_\blambda$ then there is usually a ``gap'' between 
the almost sure values of $\dim_\Phi F_\blambda$ and $\dim_\Phi \mu_\blambda$.
However, the set of attainable values for $\dim_\Phi \mu_\blambda$ is always an interval, contains the minimal (or maximal) value, and
contains values arbitrarily close to $\infty$ (respectively, $0$).

Section \ref{sec:phidim} presents the necessary background on $\Phi$ dimensions and
Section \ref{sec:construction} presents the details of the random model.
Section \ref{sec:G_functions} proves technical results about the machinery necessary to obtain the $\Phi$ dimensions where the main tools
are auxiliary functions $G(\theta)$ (for the upper dimensions) and $G'(\theta)$ (for the lower dimensions).
We prove that the $\Phi$ dimension is an extremal value of $G$ (or $G'$) in Theorem \ref{SupThm} and 
in Theorem \ref{thm:G(t)=t} show that it is the only place where $G$ (or $G'$) crosses the diagonal.
We also give some conditions under which the dimension is a continuous
function of the parameters of the model in Propositions \ref{prop:finite_GContinuous} and \ref{prop:single_continuous}.

The main results of the paper are given in Sections \ref{sec:range} and \ref{sec:independentrange} which study the set of attainable a.s. values for $\udim_\Phi \mu_\blambda$ and $\odim_\Phi \mu_\blambda$.
Section \ref{sec:range} concentrates on models where the probability weights are allowed to depend on the random scaling factors (the ``dependent'' case) and here 
Theorems \ref{thm:upper_dependent} and \ref{thm:lower_dependent} indicate that the entire possible range is attainable.
Specifically, with $D$ the almost sure Hausdorff dimension of $F_\blambda$, we have that $\udim_\Phi \mu_\blambda$ spans $(0,D]$ and $\odim_\Phi \mu_\blambda$ spans $[D,\infty)$.
As a consequence, we get that $\udim_\Phi F_\blambda = \dim_H F_\blambda = \odim_\Phi F_\blambda$ almost surely, as recorded in Corollary \ref{cor:asdimF}.

By contrast, in Section \ref{sec:independentrange} we show that when the probability weights are independent of $F_\blambda$ (the ``independent case''), there is generally a gap between $D$ and the attainable dimensions. 
Theorem \ref{thm:general_gap_condition} gives a sufficient condition for this to occur and Corollary \ref{cor:gap_finite} is a simplified version
for when the system is finite.
Additionally, Propositions \ref{prop:min_single} and \ref{prop:max_single} state that when the probabilities are independent of $F_\blambda$, both $\inf \odim_\Phi \mu_\blambda$ and $\sup \udim_\Phi \mu_\blambda$ are achieved by using single sets of probability weights.
Furthermore, the sets of achievable $\udim_\Phi \mu_\blambda$ and $\odim_\Phi \mu_\blambda$ are intervals and fully attainable by single sets of probability weights, see Theorem \ref{thm:singlefull}.
Finally, in Section \ref{sec:algorithm} we specialize to the situation where there are finitely many sets of IFS each with only two similarities.
We give an algorithm to determine the minimal possible value for $\odim_\Phi \mu_\blambda$ and also give an explicit formula for this value
when there are only two IFS.

We mention that all of our results are given in the context of subsets of $[0,1]$ for convenience and simplicity.
This restriction is not necessary and it is simple to extend the results to compact subsets of $\R^n$.

\section{Assouad-like Dimensions: Definitions and Basic Properties}

\label{sec:phidim}

Let $X$ be a bounded metric space. We will write$\ B(x,R)$ for the open ball
centred at $x$ belonging to the bounded metric space $X$ and radius $R$, and
denote by ${\mathcal{N}}_{r}(E)$ the least number of open balls of radius $r$
required to cover $E\subseteq X$. By a measure on $X$ we mean a Borel
probability measure supported on $X$.

\begin{definition} \label{def:phi}
A \textbf{dimension function} is a map $\Phi :(0,1)\rightarrow \mathbb{R}^{+}$ with the property that 
$t^{1+\Phi (t)}$ decreases to $0$ as $t$ decreases to $0$ and is doubling near zero, that is, there are $T,C > 0$
so that $(2 t)^{1 + \Phi(2 t)} \le C \cdot t^{1 + \Phi(t)}$ for all $0 < t < T$.
 We will say that a dimension function $\Phi$ is \textbf{large} if 
\begin{equation*}
    \Phi (t)=H(t)\frac{\log \left| \log t\right|}{\left| \log t\right|}
\end{equation*}
where $H(t)\rightarrow \infty$ as $t\rightarrow 0$ and \textbf{small} if
(with the same notation) $H(t)\rightarrow 0$ as $t\rightarrow 0$.
\end{definition}

Examples of large dimension functions include the constant functions $\Phi(t)=c>0$, as well as the functions 
$1/|\log t|^{\alpha}$ and $(\log\left| \log t\right|)^{1/\alpha}/\left| \log t\right|$
for $\alpha <1$. The function $\Phi =0$ and the functions $1/|\log
t|^{\alpha}$ and $(\log \left| \log t\right|)^{1/\alpha
}/\left| \log t\right| $ for $\alpha >1$ are some examples of small
dimension functions. We are primarily interested in large dimension
functions in this paper.
The doubling assumption is met, for instance, if $\Phi$ is non-decreasing and bounded near $t=0$.

\begin{definition}
Let $\Phi$ be a dimension function.

(i) The \textbf{upper} and \textbf{lower} $\Phi$\textbf{-dimensions} of
the set $E\subseteq X$ are given, respectively, by 
\begin{equation*}
          \odim_{\Phi}E=   \inf \left\{ \begin{array}{c}
              d:(\exists C_{1},C_{2}>0)(\forall 0<r\leq R^{1+\Phi (R)}<R<C_{1})\text{} \\ 
           {\mathcal{N}}_r(B(z,R)\bigcap E)\leq C_{2}\left( \frac{R}{r}\right)^{d}\text{ }\forall z\in E\end{array} \right\}
\end{equation*}
and
\begin{equation*}
   \udim_{\Phi}E=\sup \left\{ \begin{array}{c}
           d:(\exists C_{1},C_{2}>0)(\forall 0<r\leq R^{1+\Phi (R)}<R<C_{1})\text{} \\ 
         {\mathcal{N}}_r(B(z,R)\bigcap E)\geq C_{2}\left( \frac{R}{r}\right)^{d}\text{ }\forall z\in E
       \end{array}
   \right\}.
\end{equation*}

(ii) The \textbf{upper and lower} $\Phi$\textbf{-dimensions} of the measure $\mu$ on $X$ are given, respectively, by
\begin{equation*}
  \odim_{\Phi}\mu =\inf \left\{ \begin{array}{c}
            d:(\exists C_{1},C_{2}>0)(\forall 0<r<R^{1+\Phi (R)}\leq R\leq C_{1}) \\ 
         \frac{\mu (B(x,R))}{\mu (B(x,r))}\leq C_{2}\left( \frac{R}{r}\right)^{d}\text{ }\forall x\in \supp\mu
      \end{array}
   \right\}
\end{equation*}
and 
\begin{equation*}
   \udim_{\Phi}\mu =\sup \left\{ \begin{array}{c}
     d:(\exists C_{1},C_{2}>0)(\forall 0<r<R^{1+\Phi (R)}\leq R\leq C_{1}) \\ 
    \frac{\mu (B(x,R))}{\mu (B(x,r))}\geq C_{2}\left( \frac{R}{r}\right)^{d}\text{ }\forall x\in \supp\mu
   \end{array}
\right\}.
\end{equation*}
\end{definition}

In the special case of the small dimension function $\Phi =0$, the upper and lower $\Phi$-dimensions are known as the \textbf{upper} and 
\textbf{lower Assouad dimensions} (or the Assouad dimension and lower dimension) and
denoted $\dim_{A}\Theta $ and $\dim_{L}\Theta $ respectively, for sets or measures $\Theta$. 
The Assouad dimensions of sets were first introduced in \cite{A2} and have been extensively studied since. 
We refer the reader to Fraser's book, \cite{Fraserbook}, and the many references cited there-in. 
The Assouad dimensions of measures are also known as the upper and lower regularity dimensions and were studied in \cite{FH,KL, KLV}. 
The Assouad dimensions are important both because of their extremal properties (see the Proposition below), as well as the fact that a measure $\mu$ is doubling if and only if $\dim_{A}\mu <\infty$ (\cite{FH}) and uniformly perfect if and only if $\dim_{L}\mu >0$ (\cite{KL}).
We warn the reader that some sources (like \cite{BRT23}) use a slightly different definition and instead require that $r = R^{1+\Phi(R)}$, which
leads to subtle differences.

The $\Phi$-dimensions of sets were introduced by Fraser and Yu in \cite{FY} who focussed on the \textbf{upper} and 
\textbf{lower} $\theta$\textbf{-Assouad spectrum} which are essentially the special case of $\Phi$
dimensions arising from the large dimension functions $\Phi_{\theta}=1/\theta -1$ for $0<\theta <1$. 
From their limiting behaviour one obtains the \textbf{upper} and \textbf{lower quasi-Assouad dimensions} of sets,
introduced in \cite{LX} and defined as
\begin{equation*}
   \dim_{qA}E=\lim_{\theta \rightarrow 1}\odim_{\Phi_{\theta}}E 
        \text{ and }
   \dim_{qL}E=\lim_{\theta \rightarrow 1}\udim_{\Phi_{\theta}}E.
\end{equation*}
General $\Phi$-dimensions for sets were first extensively investigated in \cite{GHM}.

Motivated by these dimensions for sets, the upper and lower quasi-Assouad dimensions of measures (defined analogously) were studied in \cite{HHT, HT} and the most general $\Phi$-dimensions of measures were introduced in \cite{HH}.

The `large' versus `small' dimension dichotomy corresponds to the threshold phenomena which is observed in the almost sure values of the $\Phi$ dimensions of certain random Moran sets and measures in \cite{HM2}, and as well as in other random iid constructions found in \cite{GHM2} and \cite{Tr2}.

One obvious property of the $\Phi$ dimensions is their ordering.

\begin{proposition}
If $\Phi (t) \leq \Psi (t)$ for all $t>0$, then whether $\Theta $ is a set or measure
\begin{equation*}
    \odim_{\Psi}\Theta \leq \odim_{\Phi}\Theta \ \text{ and }\  \udim_{\Phi}\Theta \leq \udim_{\Psi}\Theta.
\end{equation*}
\end{proposition}

Other well known relationships between these dimensions are stated below.
Proofs can be found in \cite{FY, GHM, HH} and the references cited therein.

\begin{proposition}
\label{prop:dimensionproperties} 
Let $E$ be a set, $\mu$ be a measure and $\Phi$ be any dimension function. 
We denote by $\udim_{B}$ ($\odim_{B}$) the lower (upper) box dimension, by $\dim_{H}$ the Hausdorff dimension
 and by $\odim_{M}$ the upper Minkowski dimension.

(i) Then
\begin{equation*}
   \dim_{L}E\leq \udim_{\Phi}E\leq \udim_{B}E\leq \odim_{B}E\leq \odim_{\Phi}E\leq \dim_{A}E,
\end{equation*}
  and 
\begin{eqnarray*}
       \dim_{L}\mu &\leq &\udim_{\Phi}\mu \leq \dim_{H}\mu \leq \dim_{H}\supp\mu \\
                    &\leq &\odim_{\Phi}\ \supp\mu \leq \odim_{\Phi}\mu \leq \dim_{A}\mu.
\end{eqnarray*}

(ii) If $\mu$ is a doubling measure, then 
\begin{equation*}
        \udim_{\Phi}\mu \leq \udim_{\Phi}\supp\mu.
\end{equation*}

(iii) If $\Theta$ is either a set or measure and $\lim_{t\rightarrow 0}$ $\Phi (t)=0$, then 
\begin{equation*}
    \udim_{\Phi}\Theta \leq \dim_{qL}\Theta \ \text{ and }\ \dim_{qA}\Theta \leq \odim_{\Phi}\Theta.
\end{equation*}

(iv) If $\Phi (t)\rightarrow \infty$ as $t\rightarrow 0$, then $\odim_{B}E=\odim_{\Phi}E$ and $\odim_{\Phi}\mu =\odim_{M}\mu$.
\end{proposition}

In general, the inequalities above can be strict.
Furthermore, part (iv) of the proposition is not true for the lower dimension as can be seen by the following example.  Take $E = C \cup [2,3]$ with $C \subset [0,1]$ the standard $1/3$ Cantor set.
Then $\udim_B E = 1$ but $\udim_\Phi E = \ln(2)/\ln(3)$.

Another interesting fact (see \cite{HH}) is that given any $0<\alpha <\beta<1$, there is a central Cantor set $E\subset [0,1]$ such that if $\mu$ is the uniform Cantor measure supported on $E$, then
\begin{equation*}
   \left\{ \odim_{\Phi}E:\lim_{x\rightarrow 0}\Phi =0\right\}
       =\left\{ \odim_{\Phi}\mu :\lim_{x\rightarrow 0}\Phi =0\right\}
             =[\alpha ,\beta ]=[\dim_{qA}E,\dim_{A}E].
\end{equation*}

The `large' versus `small' dimension threshold is of interest because in
several random iid constructions (see \cite{GHM2, HM2, Tr2}) it has been
observed that the almost sure values of the large $\Phi$ dimensions
coincide with the quasi-Assouad dimensions while the almost sure small $\Phi$ dimensions coincide with the Assouad dimensions.

\section{The random Moran sets and measures}

\label{sec:construction}

\subsection{Description of the Model}

The model for the random Moran sets and measures that we study in this paper
is the same as that studied in \cite{HM, HM2}. 
We outline the description here for the convenience of the reader.

Fix the integer $K \ge 2$ and constants $A,B, \tau > 0$ with $0 < K \cdot A \le B < 1$ and $\tau \le (1-B)/K$.

We start with a collection $\cI_u = \{ S_1(u), S_2(u), \ldots, S_K(u) \}$, indexed by $u \in \cA$, of IFSs consisting of $K$ similarities
acting on $[0,1]$.
We denote the scaling ratio of $S_i(u)$ by $a(i,u)$ and assume that for all $u \in \cA$
\begin{equation} \label{eq:scaling_constraints}
     A \le \min_{1\le i \le K} a(i,u) < \sum_{i=1}^K a(i,u) \le B.
\end{equation}
We further assume that 
\begin{equation*}
     dist(S_{i}(u)[0,1],\,S_{j}(u)[0,1])\geq \tau \text{ for all } i\neq j \text{ and all } u,
\end{equation*}%
so each $\cI_u$ is a strongly separated IFS.

Next, for each $v \in \cW$, we have a probability weight vector with components $p(i,v)$ for $i=1,2,\ldots, K$. 
For each $v \in \cW$ these satisfy
\begin{equation} \label{eq:probability_constraints}
  \sum_{i=1}^K p(i,v) = 1  \quad \text{ and } \quad p(i,v) \ge 0 \text{ for all } i=1,2,\ldots, K.
\end{equation}
The set $\cA$ indexes all the IFSs that could potentially be used in our model and $\cW$ all the probability weights that could potentially be used.

\medskip

We set $\Lambda = \cA \times \cW$ and place a probability measure $\P$ on $\Lambda$.
For simplicity, for $\lambda = (u,v) \in \Lambda$ we will often write $p(i,\lambda)$ for $p(i,v)$ and $a(i,\lambda)$ instead of $a(i,u)$.

\medskip

For technical reasons we require that there be some $t > 0$ so that $\ee_\lambda( p(i,\lambda)^{-t}) < \infty$ for each $i=1,2,\ldots, K$.
In particular, this guarantees that the probability that any $p(i,\lambda) = 0$ or $1$ is zero.
In the special case when $\cW$ is a single point  (we call this the \emph{single probability} case) the
requirement simply says $p(j) \neq 0$ for any $j=1,...,K$,  and we refer to these probabilities as the \emph{valid probabilities} in this case.

Our random Moran sets and measures are constructed through an iid sequence 
of choices of scaling ratios and probability weights,
meaning that $\blambda= (\blambda_{n})\in \Lambda^\N$ is chosen according to the infinite product measure $\P^{\mathbb{N}}$. 
This ensures that all pairs $\{a(i,\blambda_{n}),p(j,\blambda_{n})\}$ are independent of $\{a(k,{\blambda}_{m}),p(l,{\blambda}_{m})\}$ for $n\neq m$.

\subsubsection{Dependent/Independent probabilities}

A natural and particularly interesting example of our setup is when $\Lambda$ 
is a product of probability spaces, that is $\P$ is a product of probability measures on $\cA$ and $\cW$ respectively.
Then the contraction factors, which depend only on $u$, and the probabilities, which depend only on 
$v$, are independent of each other. 
We refer to this special situation as the \emph{independent probabilities} case to distinguish it from the 
\emph{dependent probabilities} case. 
As we will see later in this paper, the dimensional properties of the random measures (defined below) are quite
different in the two cases.

\subsubsection{Random Moran set}

Given ${\blambda}\in \Lambda^\N$ we put 
\begin{equation*}
      I_{i_{1},...,i_{n}}({\blambda})=S_{i_{1}}(\blambda_{1})\circ \cdots \circ S_{i_{n}}(\blambda_{n})[0,1],
\end{equation*}
a \textit{step (or level)} $n$\textit{\ Moran interval}. 
The step $n+1$ Moran intervals, $I_{i_{1},..i_{n},j}({\blambda})$ for $j=1,...,K$, are called the \textit{children} of the step $n$ Moran interval $I_{i_{1},...,i_{n}}({\blambda})$. 
We note that the length of $I_{i_{1},...,i_{n}}({\blambda})$, denoted 
$\left| I_{i_{1},...,i_{n}}({\blambda})\right|$, is equal to the product $a(i_{1},{\blambda}_{1})\cdots a(i_{n},{\blambda}_{n})$ 
and the distance between any two children of $I_{i_{1},...,i_{n}}({\blambda})$ is at least $\tau \left| I_{i_{1},...,i_{n}}({\blambda})\right|$. 
Corresponding to each $\blambda\in \Lambda^\N$ is the \textit{random Moran set (or attractor)} 
\begin{equation*}
     F_\blambda=\bigcap_{n}\bigcup_{i_{1},...,i_{n}\in \{1,...,K\}}I_{i_{1},...,i_{n}}(\blambda).
\end{equation*}
When ${\blambda}\in \Lambda^\N$ is clear we may write $I_{n}(x)$
for the unique step $n$ Moran interval that contains $x\in F_{\blambda}$. 
We remark that the assumptions on the contraction factors ensure that 
\begin{equation*}
    A^{n}\leq \frac{\left| I_{N}(x)\right|}{\left| I_{N+n}(x)\right|}\leq B^{n}.
\end{equation*}

\subsubsection{Random Measure}

The \textit{random probability measure} $\mu_\blambda$ is almost surely supported on the random set $F_{\mathbf\lambda}$ and is
defined by the rules $\mu_\blambda([0,1])=1$ and if $\nu=(i_{1},...,i_{n})$ then for each $j\in \{1,...,K\}$ 
\begin{equation*}
   \mu_\blambda(I_{\nu ,j}(\blambda))=\mu_\blambda(I_{\nu}(\blambda))p(j,\blambda_{n+1})=
        p(i_{1},\blambda_{1})\cdot \cdot \cdot p(i_{n},\blambda_{n})p(j,\blambda_{n+1}).
\end{equation*}
Notice that the random Moran set $F_\blambda$ depends only on the IFSs $\{\cI_\blambda\}$, while the random measure depends on both the IFSs
and the probabilities $\{\cR_\blambda\}$. 
By varying the probabilities we obtain a family of random measures all of which are supported on $F_\blambda$ almost surely.

For those familiar with V-variable fractals (c.f. [2]), $F_\blambda $ is a 1-variable fractal and $\mu_\blambda$ a 1-variable fractal measure.

\section{Determining the large $\Phi$-dimensions}
\label{sec:G_functions}

\subsection{The basic dimension theorem}

As shown in Lemma 2 of \cite{HM2} (see also arXiv:2207.14654 for more details), the uniform separation satisfied by the Moran sets makes it easy to establish that the $\Phi$ dimensions 
(for both large and small $\Phi$) of both the random set $F_\blambda$ and the random measure $\mu_\blambda$ can be computed
using only the lengths and measures of the Moran intervals.  
We remark that the doubling assumption from Def. \ref{def:phi} is not needed to establish this for the upper dimensions. 
Using this Lemma, a simple formula was found for the $\Phi$ dimensions of the random measures for small dimension functions $\Phi$ in \cite[Theorem 21]{HM2}.

The problem of determining the $\Phi$ dimensions for large $\Phi$ is much more intricate, however, because of the depth of scales which must be considered. 
In \cite{HM2} it was shown that we can understand these dimensions in terms of properties of functions, denoted $G$ and $G'$, which are
ratios of expected values of logarithms of probabilities to logarithms of scaling ratios. 
As we make extensive use of the functions $G$ and $G'$ in this paper, we state their definitions here.

\textbf{Notation:} Given our model as described above, we define the following functions with common 
domain $(\theta,\lambda) \in [0,\infty) \times \Lambda$
\begin{align*}
   Y(\theta ,\lambda) &=\log p(m,\lambda), & \qquad Y'(\theta ,\lambda) &=\log p(m',\lambda) \\
    Z(\theta ,\lambda) &=\log a(m,\lambda),& \qquad Z'(\theta ,\lambda)  &=\log a(m',\lambda)
\end{align*}%
where $m=m(\theta ,\lambda) \in \{1,2,\ldots, K \}$ is the minimal index which satisfies the condition 
\begin{equation}  \label{eq:m_definition}
   \frac{a(m,\lambda)^{\theta}}{p(m,\lambda)}=
         \max \left( \frac{a(j,\lambda)^{\theta}}{p(j,\lambda)}:j=1,...,K\right) 
\end{equation}
and $m'=m'(\theta ,\lambda)$ is the minimal index which satisfies 
\begin{equation} \label{eq:mprime_definition}
   \frac{a(m',\lambda)^{\theta}}{p(m',\lambda)}=
   \min\left( \frac{a(j,\lambda)^{\theta}}{p(j,\lambda)}:j=1,...,K\right).
\end{equation}
Then define 
\begin{equation} \label{eq:G_definition}
    G(\theta)=\frac{\ee_\lambda(Y(\theta ,\lambda))}{\ee_\lambda(Z(\theta ,\lambda))}
          \quad \text{ and } \quad
   G'(\theta)=
       \frac{\ee_\lambda(Y'(\theta ,\lambda))}{\ee_\lambda(Z'(\theta ,\lambda))}
\end{equation}
where the notation $\ee_\lambda$ means that the expectation is taken with respect to  $\lambda \in \Lambda$ using the measure $\P$. 
However, as we will only be taking expectations with respect to $\lambda$ we will usually omit the subscript.
Note that with $\blambda \in \Lambda^\N$, each family, $\{Y(\theta ,\blambda_n)\}$, 
$\{Z(\theta ,\blambda_n)\}$, $\{Y'(\theta ,\blambda_n)\}$ and $\{Z'(\theta ,\blambda_n)\}$, is an independent and identically distributed sequence of random variables.

Here is simple fact about these functions.

\begin{proposition}
\label{GBded}
Assume the random model as given above.

(i) All four quantities $\left| \ee_\lambda(Y(\theta,\lambda))\right| $, $\left| \ee_\lambda(Y'(\theta,\lambda))\right| $, $\left| \ee_\lambda(Z(\theta,\lambda))\right|$ and 
$\left| \ee_\lambda(Z'(\theta,\lambda))\right|$ are bounded above and bounded away from $0$.

(ii) Both $G(\theta)$ and $G'(\theta)$ are bounded above and bounded away from $0$.
\end{proposition}

\begin{proof}
(i) Since $A\leq a(i,\lambda)\leq B$ for all $i$ and $\lambda $, we have $|\log (B)|\leq |\log (a(i,\lambda))|\leq |\log (A)|$. 
That proves the claim about $\left| \ee_\lambda(Z(\theta,\lambda))\right| $ and 
$\left| \ee_\lambda (Z'(\theta,\lambda))\right|$.

As the probabilities $p(i,\lambda)$ can only take on the value $1$ with probability zero, we know that 
\begin{equation*}
    \P(\{\lambda :\max_{i}p(i,\lambda)=1\})=0.
\end{equation*}
Thus there is a $\delta >0$ so that if $\Lambda_\delta=\{\lambda \in \Lambda :\max_{i}p(i,\lambda)\leq 1-\delta \}$ 
then $\P(\Lambda_\delta)\geq \eta $ for some $\eta >0$. 
Notice that $Y(\theta,\lambda)\leq \log (1-\delta)$ for any $\lambda \in \Lambda_\delta$ and all $\theta$. 
Using this we get that 
\begin{align*}
\left| \ee(Y(\theta,\lambda))\right| & \geq \int_{\{ Y(\theta,\lambda)\leq \log (1-\delta) \} } \hskip -2 cm -Y(\theta,\lambda)\ d\P(\lambda) \\
                   & \geq \int_{\Lambda_{\delta}}\left| \log (1-\delta)\right| \ d\P(\lambda)\geq 
                                     \eta \left| \log (1-\delta)\right| \geq \delta \eta.
\end{align*}
The argument is the same for $\left| \ee(Y'(\theta,\lambda))\right|$.

By our assumptions on the probabilities, there exists some $t<0$ such that $\ee(p(i)^{-t})<\infty $ for $i=1,2,\ldots ,K$. Since 
\begin{equation*}
     \lim_{x\rightarrow 0^{+}}\frac{-\log (x)}{x^{-t}}=0
\end{equation*}
there is an $0<\eta <1$ so that whenever $0<p\leq \eta $ we have $|\log(p)|\leq p^{-t}$. 
But then, for all $\theta $, 
\begin{align*}
   |\ee(Y(\theta,\lambda))|& \leq \int \sum_{i=1}^K |\log (p(i,\lambda))|\ d\P(\lambda) \\[10pt]
         & \leq \sum_{i=1}^K\left( \int_{p(i,\lambda)\leq \eta}p(i,\lambda)^{-t}\ d\P(\lambda)+
                           \int_{p(i,\lambda)>\eta}-\log (p(i,\lambda))\ d\P(\lambda)\right)  \\[10pt]
                   & \leq \sum_{i=1}^K \ee(p(i,\lambda)^{-t})+\sum_{i=1}^K -\log (\eta) < \infty.
\end{align*}
The argument is the same for $|\ee(Y'(\theta,\lambda))|$.

(ii) The boundedness of $G$ and $G'$ (above and below) follow immediately from (i).
\end{proof}

For the convenience of the reader we state here the theorem on the $\Phi$
dimensions of the random measures for large dimension functions $\Phi$ found in \cite[Thm 11]{HM2}.

\begin{theorem}
\label{MainDimThm} 
Let $\mu_\blambda$ be the random measure.

(i) If $G(\psi)<\psi $ (or $G(\psi)\geq \psi)$, then there is a set $\Gamma_{\psi}\subseteq \Lambda$, of full measure, such that 
\begin{equation*}
           \odim_{\Phi}\mu_\blambda\leq \psi\  (\text{respectively}, \odim_{\Phi}\mu_\blambda\geq \psi)
\end{equation*}
for all large dimension functions $\Phi$ and all $\lambda \in \Gamma_\psi$.

(ii) If $G'(\psi)>\psi $ (or $G'(\psi)\leq \psi)$, then there is a set $\Gamma_{\psi}\subseteq \Lambda$, of full measure, such that 
\begin{equation*}
    \udim_{\Phi}\mu_\blambda\geq \psi \ (\text{respectively},  \udim_{\Phi}\mu_\blambda\leq \psi).
\end{equation*}
\end{theorem}

\begin{corollary}
(i) $\odim_{\Phi}\mu_\blambda=\alpha $ almost surely if and only if $G(\psi)<\psi $ for all $\psi >\alpha $ and $G(\psi)\geq \psi$
for all $\psi <\alpha$.

(ii) If there exists $\alpha $ such that $G(\alpha)=\alpha $ and $G(\psi)<\psi$ if $\psi >\alpha$, then 
$\odim_{\Phi}\mu_{\lambda}=\alpha $ a.s.
\end{corollary}

There is an analogous corollary for the lower $\Phi$ dimensions.

\subsection{Dimension is the supremum of $G$ and infimum of $G'$}

In this subsection we derive a slightly different way of finding the almost sure $\Phi$-dimensions of $\mu_\blambda$, with the
upper dimension as the supremum of $G(\theta)$ and the lower dimension as the infimum of $G'(\theta)$.
Furthermore we also show that the upper and lower almost sure dimensions are precisely the solutions to $G(t) = t$ and $G'(t) = t$, respectively.

\begin{theorem} \label{SupThm}
For all large $\Phi$, almost surely $\odim_\Phi\mu_\blambda=\sup_{\theta \geq 0}G(\theta)$ and 
$\udim_\Phi \mu_\blambda=\inf_{\theta \geq 0}G'(\theta)$.
\end{theorem}

\begin{proof}
We have seen in Prop. \ref{GBded} that $G(\theta)$ is bounded, so let $d=\sup_{\theta}G(\theta)$.

For any $\epsilon >0$, we know that $G(d+\epsilon)\leq \sup_{\theta}G(\theta)=d<d+\epsilon $, 
so by Theorem \ref{MainDimThm} (from \cite[Thm.11]{HM2}) we have $\odim_{\Phi}\mu_\blambda\leq d+\epsilon$
almost surely. This implies that $\odim_{\Phi}\mu \leq d$ almost surely.

The proof of the reverse inequality is a slight modification of the argument from the proof of Theorem 11(ii) in \cite{HM2}. 
We provide a sketch for the benefit of the reader.

For each $\epsilon ,\delta \in (0,d)$ there are $\theta_{\epsilon} \ge 0$ with $G(\theta_{\epsilon})>d-\epsilon $ and a set $\Omega_{\delta}^{\epsilon} \subseteq \Lambda^\N$
of full measure so that for all $\blambda \in \Omega_{\delta}^{\epsilon}$ and large enough $N$ we have 
\begin{equation*}
    \left| \frac{\sum_{i=N+1}^{N+k}Y(\theta_{\epsilon},\blambda_i)}{\sum_{i=N+1}^{N+k}Z(\theta_{\epsilon},\blambda_i)}-
    G(\theta_{\epsilon})\right| \leq \delta ,\quad \forall k \geq \zeta_N^H :=
   \dfrac{ H(B^n) \log(N |\log B| )}{|\log A|}
\end{equation*}
and thus 
\begin{equation*}
       \frac{\sum_{i=N+1}^{N+k}Y(\theta_{\epsilon},\blambda_i)}{\sum_{i=N+1}^{N+k}Z(\theta_{\epsilon},\blambda_i)}
           \geq G(\theta_{\epsilon})-\delta >d-\epsilon-\delta,
\end{equation*}
which implies that
\begin{equation}    \label{eq:Gthetaepsilon}
  \sum_{i=N+1}^{N+k}Y(\theta_{\epsilon},\blambda_i)<(d-\epsilon -\delta)\sum_{i=N+1}^{N+k}Z(\theta_{\epsilon},\blambda_i).
\end{equation}
Now, given $\blambda \in \Omega_\delta^\epsilon$, consider the Moran intervals which arise by choosing the child with index $m(\theta_\epsilon,\blambda_n)$ at step $n$. 
Call the $n$-step interval obtained this way $I_{n}=I_{n}(\theta_\epsilon,\blambda)$ and note that these form a nested sequence of Moran intervals.
Then since 
\[
    \mu_\blambda(I_{N+i+1}) = \mu_\blambda(I_{N+i}) \cdot p(m(\theta_\epsilon,\blambda_{N+i+1}),\blambda_{N+i+1})
\]
and
\[
   |I_{N+i+1}| = |I_{N+i}| \cdot a(m(\theta_\epsilon,\blambda_{N+i+1}),\blambda_{N+i+1}),
\]
we see that (\ref{eq:Gthetaepsilon}) implies that 
\begin{equation*}
    \frac{\mu_\blambda(I_{N})}{\mu_\blambda(I_{N+k})}>\left( \frac{|I_{N}|}{|I_{N+k}|}\right)^{d-\epsilon -\delta}
\end{equation*}
for all $\blambda \in \Omega_{\delta}^{\epsilon}$, $k\geq \zeta_{N}^{H}$ and $N$ sufficiently large. 
Thus $\odim_{\Phi}\mu_\blambda \geq d$ almost surely.

\smallskip

The proof for the lower dimension is similar.
\end{proof}

\begin{corollary}
$\udim_{\Phi}\mu_\blambda>0$ and $\odim_{\Phi}\mu_\blambda<\infty $ almost surely.
\end{corollary}

\begin{proof}
This is immediate from   Prop. \ref{GBded} and Theorem \ref{SupThm}.
\end{proof}

The next technical lemma will be useful to us.  
It roughly indicates that the choices of indices in (\ref{eq:m_definition}) and (\ref{eq:mprime_definition}) are the extremes.

\begin{lemma}
\label{lem:H-function}
Let $\chi _{n}:\Lambda \to \{1,2, \ldots, K\}$ be iid rv's  and let 
\begin{equation*}
    H(\chi )=\frac{\ee(\log p(\chi ,\lambda ))}{\ee(\log a(\chi,\lambda ))}.
\end{equation*}
Then 
\begin{equation*}
    \inf_{\theta }G^{\prime }(\theta )\leq H(\chi )\leq \sup_{\theta }G(\theta ).
\end{equation*}
\end{lemma}

\begin{remark}
Of course, in the special case that $\chi (\lambda )=m(\theta ,\lambda )$
(or, $m^{\prime }(\theta ,\lambda )$) then $H(\chi )=G(\theta )$ (resp., 
$G^{\prime }(\theta )$).
\end{remark}

\begin{proof}
Given $\lambda ,$ consider the Moran intervals which arise by choosing the $\chi _{n}(\lambda )$'th - child at step $n$. 
Notice these Moran intervals are nested and we will denote the interval arising at step $n$ as 
$I_{n}^{(\lambda )}$. 
For $n>N$ we have
\begin{equation*}
   \frac{\mu _{\lambda }(I_{N}^{(\lambda )})}{\mu _{\lambda }(I_{n}^{(\lambda)})}
      =\left( \prod_{i=N+1}^{n}p(\chi _{i}(\lambda ),\lambda )\right)^{-1}
\end{equation*}
and 
\begin{equation*}
  \frac{\left\vert I_{N}^{(\lambda )}\right\vert }{\left\vert I_{n}^{(\lambda)}\right\vert }
   =\left( \prod_{i=N+1}^{n}a(\chi _{i}(\lambda),\lambda )\right) ^{-1}.
\end{equation*}

Let $\varepsilon >0$. For $N$, $n>N$ sufficiently large we have 
\begin{equation*}
 \frac{\mu _{\lambda }(I_{N}^{(\lambda )})}{\mu _{\lambda }(I_{n}^{(\lambda)})}
  \leq 
   \left( \frac{\left\vert I_{N}^{(\lambda )}\right\vert }{\left\vert I_{n}^{(\lambda )}\right\vert }\right) ^{\odim_{\Phi }\mu
+\varepsilon },
\end{equation*}
thus for all $\lambda$
\begin{eqnarray*}
    -\log \prod_{i=N+1}^{n}p(\chi _{i}(\lambda ),\lambda ) &\leq &
         -(\odim_{\Phi }\mu _{\lambda }+\varepsilon )\log \prod_{i=N+1}^{n}a(\chi _{i}(\lambda ),\lambda ).
\end{eqnarray*}%
Equivalently,
\begin{equation*}
 \frac{\sum_{i=N+1}^{n}\log p(\chi _{i}(\lambda ),\lambda )}{\sum_{i=N+1}^{n}\log a(\chi _{i}(\lambda ),\lambda )}
         \leq \odim_\Phi \mu_\blambda + \epsilon.
\end{equation*}
But, as $n\rightarrow \infty $ the left hand side of the inequality above converges almost surely to $H(\chi )$ while the right hand side is equal to $\sup_\theta G(\theta) + \epsilon$ almost surely. 
Since the inequality holds for all  $\varepsilon >0$ we conclude that $H(\chi )\leq \sup_{\theta }G(\theta )$.

A symmetric argument gives the proof of the lower bound.
\end{proof}

Our next result uses this lemma to show that $\odim_\Phi \mu$ is also a solution to $G(t) = t$.

\begin{theorem} \label{thm:G(t)=t}

(i) $\sup_{\chi }H(\chi )=\sup_{\theta }G(\theta )=\odim_\Phi\mu_\blambda $ a.s.

(ii) $t= \sup_{\chi }H(\chi )$ if and only if $G(t)=t$

(iii) $\inf_\chi H(\chi) = \inf_\theta G'(\theta) = \udim_\Phi \mu_\blambda$ a.s.

(iv) $t = \inf_\chi H(\chi)$ if and only if $G'(t) = t$
 
\end{theorem}

\begin{proof}
Statements (i) and (iii) are immediate from Lemma \ref{lem:H-function}.

(ii) First, suppose $t=G(t)$. 
Then $t\leq \sup_{\theta }G(\theta)=\sup_{\chi }H(\chi )$ so we only need prove the other inequality.

To shorten notation, we will write $m_{n}=m(t,\lambda _{n})$ and $\chi _{n}=\chi(\lambda _{n})$.

We first observe that for any $\chi$  
\begin{equation*}
      \lim_{N}\log \prod_{n=1}^{N}a(\chi _{n})^{t/N}=
    t\lim_{N}\frac{1}{N}\sum_{n=1}^{N}\log a(\chi _{n})=
   t\ee(\log a(\chi ))\text{ for a.a. } \lambda 
\end{equation*}
and since the exponential function is continuous it follows that $\lim_{N}\prod_{n=1}^{N}a(\chi _{n})^{t/N}$ exists a.s.  
Furthermore, this limit is bounded away from $0$ as the contraction factors are bounded away from $0$. 
Similarly, $\lim_{N}\prod_{n=1}^{N}p(\chi _{n})^{1/N}$ exists a.s.

Since $t=G(t)$ we have $t\ee(\log a(m))=\ee(\log p(m))$.
Consequently, almost surely
\begin{equation*}
    \lim_{N}\log \prod_{n=1}^{N}a(m_{n})^{t/N}
  = \lim_{N}\frac{t}{N}\sum_{n=1}^{N}\log a(m_{n})
  =\lim_{N}\frac{1}{N}\sum_{n=1}^{N}\log p(m_{n})
  =\lim_{N}\log \prod_{n=1}^{N}p(m_{n})^{1/N}
\end{equation*}
and therefore 
\begin{equation*}
     \lim_{N}\prod_{n=1}^{N}a(m_{n})^{t/N}=\lim_{N}\prod_{n=1}^{N}p(m_{n})^{1/N},
\end{equation*}
equivalently,  
\begin{equation*}
      1=\lim_{N}\left( \prod_{n=1}^{N}\frac{p(m_{n})}{a(m_{n})^{t}}\right) ^{1/N}.
\end{equation*}
But by the definition of $m$, 
\begin{equation*}
     \frac{p(m_{n})}{a(m_{n})^{t}}\leq \frac{p(\chi _{n})}{a(\chi _{n})^{t}}\text{for all }\chi ,
\end{equation*}
thus
\begin{equation*}
   1\leq \lim_{N}\left( \prod_{n=1}^{N}\frac{p(\chi _{n})}{a(\chi _{n})^{t}}\right) ^{1/N}
\end{equation*}
and that implies
\begin{equation*}
   \lim_{N}\prod_{n=1}^{N}a(\chi _{n})^{t/N}\leq 
           \lim_{N}\prod_{n=1}^{N}p(\chi_{n})^{1/N}.
\end{equation*}
Taking logarithms we deduce that
\begin{equation*}
   t \ee(\log a(\chi ))=\lim_{N}\frac{t}{N}\sum_{n=1}^{N}\log a(\chi_{n})
     \leq \lim_{N}\frac{1}{N}\sum_{n=1}^{N}\log p(\chi _{n})=\ee(\log p(\chi ))\text{ a.s.}
\end{equation*}
hence  
\begin{equation*}
    t\geq \ee(\log p(\chi ))/\ee(\log a(\chi ))=H(\chi ).
\end{equation*}
Since this holds for all $\chi $ it follows that $t\geq \sup_{\chi }H(\chi )$
as we desired to show.

\medskip 

Now assume that $t=\sup_{\chi }H(\chi )=\sup_{\theta }G(\theta )$, but that $t\neq G(t)$. 
In that case, $G(t)=t-\varepsilon $ for some $\varepsilon >0$.
But then for a.a. $\lambda $ and for $N$ sufficiently large 
\begin{equation*}
   t-\varepsilon \geq 
    \frac{\frac{1}{N}\sum_{n=1}^{N}\log p(m_{n})}{\frac{1}{N}\sum_{n=1}^{N}\log a(m_{n})}
     -\frac{\varepsilon }{2}.
\end{equation*}
Hence $(t-\varepsilon /2)\sum_{n=1}^{N}\log a(m_{n})\leq \sum_{n=1}^{N}\log p(m_{n})$ 
for large enough $N,$ whence $\prod_{n=1}^{N}a(m_{n})^{t-\varepsilon /2}\leq \prod_{n=1}^{N}p(m_{n})$, 
or, equivalently,
\begin{equation*}
    \frac{\prod_{n=1}^{N}a(m_{n})^{t}}{\prod_{n=1}^{N}p(m_{n})}
        \leq
        \prod_{n=1}^{N}a(m_{n})^{\varepsilon /2}.
\end{equation*}
It follows from the definition of $m$ that for any $\chi $ we have 
\begin{equation*}
         \prod_{n=1}^{N}a(\chi _{n})^{t}\leq \prod_{n=1}^{N}a(m_{n})^{\varepsilon/2}\prod_{n=1}^{N}p(\chi _{n})
\end{equation*}
and thus
\begin{equation*}
     \frac{t}{N}\sum_{n=1}^{N}\log a(\chi _{n})\leq 
         \frac{\varepsilon /2}{N}\sum_{n=1}^{N}\log a(m_{n})+
        \frac{1}{N}\sum_{n=1}^{N}\log p(\chi _{n}).
\end{equation*}
Taking limits we deduce that 
\begin{equation*}
     t\ee(\log a(\chi ))\leq (\varepsilon /2)\ee(\log a(m))+\ee(\log p(\chi )).
\end{equation*}

Since the contraction factors are bounded away from $0,1$ there is some constant $C>0$ 
such that $\ee(\log a(m))/\ee(\log a(\chi))\geq C > 0$ for all $\chi$.
Thus 
\begin{equation*}
     t\geq 
       \frac{\varepsilon }{2}\frac{\ee(\log a(m))}{\ee(\log a(\chi ))}+
       \frac{\ee(\log p(\chi ))}{\ee(\log a(\chi ))}
    \geq   \frac{C\varepsilon }{2}+H(\chi )\text{ for all }\chi 
\end{equation*}%
and therefore $t\geq C\varepsilon /2+\sup_{\chi }H(\chi )$ which is a
contradiction.

The proof of  (iv) is similar using the obvious modifications.
\end{proof}

\begin{proposition} \label{prop:finite_GContinuous}
Suppose that there are only finitely many sets of scaling factors and finitely many probability weight vectors.
Then $\sup_\theta G(\theta)$ and $\inf_\theta G'(\theta)$ are continuous as a function of the scaling factors and probability weights and thus 
so are $\odim_\Phi \mu_\blambda$ and $\udim_\Phi \mu_\blambda$.
\end{proposition}

\begin{proof}
Using the notation from the proof of Theorem \ref{thm:G(t)=t} we know $\sup_\theta G(\theta) = \sup_\chi H(\chi)$ and $\inf_\theta G'(\theta) = \inf_\chi H(\chi)$.
Since there are only finitely many possible $\chi$ and each $H(\chi)$ is continuous in its parameters, so are $\sup_\chi H(\chi)$ and $\inf_\chi H(\chi)$.
\end{proof}

The almost sure Hausdorff dimension of the random set $F_\blambda$ arising from the contraction factors $\{a(i,\lambda)\}$ is the unique solution $x \ge 0$ to the equation (see \cite{Ham})
\begin{equation}
           \ee(\log [a(1,\lambda)^{x}+a(2,\lambda)^{x}+\cdots +a(K,\lambda)^{x}])=0.
    \label{eq:ASHausdorff}
\end{equation}
We will denote by $D$ this dimension.

We first note that for any function $\chi:\Lambda \rightarrow \{1,2,\ldots ,K\}$ we have 
\begin{align*}
         \frac{\ee(\log \left( \dfrac{a(\chi(\lambda),\lambda)^{D}}{\sum_{i=1}^{K}a(i,\lambda)^{D}}\right))}
               {\ee(\log a(\chi(\lambda),\lambda))}
     &= \frac{\ee(D\log a(\chi(\lambda),\lambda))-\ee(\log \sum_{i=1}^{K}a(i,\lambda)^{D})}{\ee(\log a(\chi(\lambda),\lambda))} \\[6pt]
     & =D-0,
\end{align*}
by the definition of $D$ and we will use this with $\chi(\lambda) = m(D,\lambda)$, 
the index which maximizes the ratio $a(i,\lambda)^{D}/p(i,\lambda)$. 
By definition of $m$, we have 
\begin{equation*}
    a(i,\lambda)^{D}\leq p(i,\lambda)\frac{a(m(D,\lambda),\lambda)^{D}}
                                                               {p(m(D,\lambda),\lambda)}
     \implies \sum_{i=1}^{K}a(i,\lambda)^{D}  \leq 
                      \left( \frac{a(m(D,\lambda),\lambda)^{D}}{p(m(D,\lambda),\lambda)}\right) 
                     \sum_{i=1}^{K}p(i,\lambda)
\end{equation*}
and so 
\begin{equation}
     p(m(D,\lambda),\lambda)\leq 
        \frac{a(m(D,\lambda),\lambda)^{D}}{\sum_{i=1}^Ka(i,\lambda)^{D}}.  
   \label{eq:prob_inequality}
\end{equation}
Intuitively, this says that the choice of index $j$ that gives the maximum ratio $a(j,\lambda)^\theta/p(j,\lambda)$  always has a matching probability 
that is no greater than the ``natural'' probability given by 
\begin{equation*}
   \frac{a(m(D,\lambda),\lambda)^{D}}{\sum_{i=1}^Ka(i,\lambda)^{D}}
\end{equation*}
(which we will observe in (\ref{eq:UpperP}) is the choice that gives the minimal measure dimension). 
This means that 
\begin{equation*}
        Y(D,\lambda)=\log (p(m(D,\lambda),\lambda))\leq 
    \log \left( \frac{a(m(D,\lambda),\lambda)^{D}}{\sum_{i=1}^Ka(i,\lambda)^{D}}\right).
\end{equation*}
Noting that $Z(D,\lambda)=\log (a(m(D,\lambda),\lambda))$, we then
have 
\begin{align}
      G(D)-D &=\frac{\ee\left( Y(D,\lambda)-\log (\dfrac{a(m(D,\lambda),\lambda)^{D}}{\sum_{i=1}^K a(i,\lambda)^{D}})\right)}
                             {\ee(Z(D,\lambda))}  \notag \\[8pt]
                 &= \frac{\ee\left( \log (p(m(D,\lambda),\lambda))-\log\left( \dfrac{a(m(D,\lambda),\lambda)^{D}}
                                  {\sum_{i=1}^Ka(i,\lambda)^{D}}\right) \right)}
                             {\ee(\log (a(m(D,\lambda),\lambda)))} \geq 0  
             \label{eq:general_gap}
\end{align}
as both numerator and denominator are non-positive. 
Thus $G(D)\geq D$, with equality possible only if and only if 
\begin{equation*}
        \log (p(m(D,\lambda),\lambda))-
        \log \left( \frac{a(m(D,\lambda),\lambda)^{D}}{\sum_{i=1}^Ka(i,\lambda)^{D}}\right) =0\text{ a.s.}
\end{equation*}
A similar argument will work to show $G'(D) \le D$ and thus we have proven the following result:

\begin{proposition}
\label{prop:general_gap} 
In general, $G(D)\geq D$ and so $\odim_\Phi \mu_\blambda \ge D$ a.s. for any large dimension function $\Phi$. 
Moreover, $G(D)=D$ if and only if 
\begin{equation*}
    p(m(D,\lambda),\lambda)=\frac{a(m(D,\lambda),\lambda)^{D}}{\sum_{i=1}^{K}a(i,\lambda)^{D}}
\end{equation*}
almost surely. 
Furthermore, if $G(D)>D$ then $\odim_{\Phi}\mu_{\blambda}>D$ almost surely.

Similarly, $G'(D) \le D$ and $\udim_\Phi \mu_\blambda \le D$ a.s. for any large dimension function $\Phi$, and $G'(D) = D$ if and only if
\begin{equation*}
    p(m'(\lambda,D),\lambda)=
    \frac{a(m'(\lambda,D),\lambda)^{D}}{\sum_{i=1}^{K}a(i,\lambda)^{D}}
\end{equation*}
almost surely. 
Furthermore, if $G'(D)< D$ then $\udim_{\Phi}\mu_{\blambda}< D$ almost surely.

\end{proposition}

\begin{corollary} \label{cor:D_bound}
Almost surely we have $\udim_\Phi \mu_\blambda \le \dim_H F_\blambda \le \odim_\Phi \mu_\blambda$.
\end{corollary}

\section{The range of dimensions for dependent probabilities}
\label{sec:range}

In this section and the next we study the dimensional properties of all the random
measures which arise on a given random set. 
That is, having fixed a set of random scaling factors which produce a random set $E$, we wish to investigate all the 
possible dimensions of the measures supported on $E$ that we could get by using random probability weights.
Since we want to allow any possible set of probability weights we will take
\[
    \cW = \{ p \in \R^K : p_i \ge 0, \sum_i p_i = 1 \}.
\]
The weights which are used in a given situation are those in the support of the measure $\P$.
In addition, this means that we fix  $\cA$, the similarities $\{I_{u}\}$, and the distribution $\cQ$ on $\cA$.
Thus we consider all probability measures $\P$ on $\Lambda = \cA \times \cW$ which have $\cQ$ as its marginal distribution on $\cA$.
Recall that we impose on $\P$ the condition $\ee_\P(p(i,\lambda)^{-t}) < \infty$ for some $t > 0$.

\medskip

Recall that we use $D$ to denote the almost sure Hausdorff dimension of the random set $F_\blambda$ (see Eq. \ref{eq:ASHausdorff}).
Notice that if $a(i,u) = a(j,u)$ almost surely for all $i,j$, then $D = -\log K/\ee(\log a(1,\lambda))$.

From Proposition \ref{prop:general_gap} we know that for any random measure $\mu_\blambda$ supported on $F_\blambda$ we have 
$\odim_{\Phi}\mu_\blambda\geq D$ and $\udim_{\Phi}\mu_\blambda\leq D$. 
In this section and the next we address the question as to whether the full interval $[D,\infty)$ (or $(0,D]$) can be attained as the
upper (resp., lower) $\Phi$-dimension of one of these random measures. 
The answer turns out to be yes when the probabilities are allowed to depend on the contraction factors, but not in general in the case of independent probabilities. 
In both situations the minimal (resp. maximal) possible measure dimension is attained by some random measure and this measure can arise from a single choice of probability weights in the independent case.

We begin with the upper dimension.

\begin{theorem} \label{thm:upper_dependent}
For any $d \in [D, \infty)$, there is a choice of probabilities $p(i,\lambda)$ so
that almost surely $\odim_\Phi \mu_\blambda = d$.
\end{theorem}

\begin{proof} With no loss of generality we assume that the indexing is such that $a(1,\lambda) \ge a(2,\lambda) \ge \cdots \ge a(K,\lambda)$
for all $\lambda \in \Lambda$.

First, suppose that $\P(\{\lambda :a(1,\lambda)>a(K,\lambda)\})>0$.
Under this assumption, there is some $\delta >0$ so that 
 $\P( \lambda: \{a(1,\lambda)\geq (1+\delta)a(K,\lambda)\})>0$. 
Choose $t\geq D$ and for $j=1,2,\ldots ,K$ set 
\begin{equation}
       p(j,\lambda)=\frac{a(j,\lambda)^t}{\sum_{i=1}^K a(i,\lambda)^t}.
  \label{eq:UpperP}
\end{equation}
Then we note that 
\begin{equation*}
   \frac{a(i,\lambda)^\theta}{p(i,\lambda)}\geq 
              \frac{a(j,\lambda)^\theta}{p(j,\lambda)}\Leftrightarrow \left( \frac{a(i,\lambda}{a(j,\lambda)}\right)^{\theta -t}\geq 1
\end{equation*}
which happens when either $a(i,\lambda)\geq a(j,\lambda)$ and $\theta \geq t$ or $a(i,\lambda)\leq a(j,\lambda)$ and $\theta \leq t$. 
This means that for $\theta \leq t$, the maximum of $a^{\theta}(i,\lambda)/p(i,\lambda)$ occurs for $i=K$ (the smallest scaling factor)
whereas for $t\leq \theta $ the maximum occurs for $i=1$ (the largest scaling factor)
and so 
\begin{equation*}
   Y(\theta,\lambda)=\begin{cases}
          \log p(K,\lambda) & \text{ if }\theta \leq t \\ 
          \log p(1,\lambda) & \text{ if }t<\theta 
             \end{cases} 
       \quad =\quad 
    \begin{cases}
             t\log a(K,\lambda)-\log (\sum_{i}a(i,\lambda)^{t}) & \text{ if }\theta \leq t \\ 
             t\log a(1,\lambda)-\log (\sum_{i}a(i,\lambda)^{t}) & \text{ if }t<\theta 
    \end{cases}
\end{equation*} 
and 
\begin{equation*}
      Z(\theta,\lambda)=\begin{cases}
            \log a(K,\lambda) & \text{ if }\theta \leq t \\ 
            \log a(1,\lambda) & \text{ if }t<\theta.
      \end{cases}
\end{equation*}
Thus 
\begin{equation*}
     G(\theta)=\begin{cases}
             t-\displaystyle\frac{\ee(\log \sum_{i}a(i,\lambda)^{t})}{\ee{\log a(K,\lambda)}} & \text{ if }\theta \leq t \\[10pt]
             t-\displaystyle\frac{\ee(\log \sum_{i}a(i,\lambda)^{t})}{\ee{\log a(1,\lambda)}} & \text{ if }t<\theta,
      \end{cases}
\end{equation*}
and so $G$ takes at most two different values.
Notice that if $t=D$ then $G(\theta)=D$ for all $\theta $ and thus $\odim_{\Phi}\mu_\blambda=D$ almost surely in this case. 
On the other hand, if $t>D$, then $\ee(\log (\sum_{i}a(i,\lambda)^{t}))<0$. 
Further since $\ee(\log (a(K,\lambda))) \le \ee(\log( a(i,\lambda)))<0$ we know that for any given $t$ the value of 
$G(\theta)$ for $\theta \le t$ is the larger of the two values. 
So
\begin{equation*}
    \odim_{\Phi}\mu_\blambda=
   t-\displaystyle\frac{\ee(\log \sum_{i}a(i,\lambda)^{t})}{\ee(\log a(K,\lambda))}.
\end{equation*}
Thus, 
\begin{align*}
   \odim_{\Phi}\mu_\blambda& =t-\frac{\ee(\log \sum_{i}a(i,\lambda)^{t})}{\ee(\log a(K,\lambda))}=
      t-\frac{t\,\ee(\log a(K,\lambda))}{\ee(\log a(K,\lambda))}-
               \dfrac{\ee(\log (\frac{\sum_{i}a(i,\lambda)^{t}}{a(K,\lambda)^t}))}{\ee(\log a(K,\lambda))} \\
       & =\frac{\ee(\log \sum_{i}\left( \frac{a(i,\lambda)}{a(K,\lambda)}\right)^{t})}{-\ee(\log a(K,\lambda))}\\
   &\geq \frac{1}{-\ee(\log a(K,\lambda))}
          \int_{a(1,\lambda) \ge (1+\delta) a(K,\lambda)}\log (\sum_{i}\left( \frac{a(i,\lambda)}{a(K,\lambda)}\right)^{t})\ d\P(\lambda) \\
         & \ge \frac{1}{-\ee(\log a(K,\lambda))}\int_{a(1,\lambda) \ge (1+\delta) a(K,\lambda)}\ \log(\left( \frac{ a(1,\lambda)}{a(K,\lambda)} \right)^t) d\P(\lambda) \\
         & \ge \frac{1}{-\ee(\log a(K,\lambda))}\int_{a(1,\lambda) \ge (1+\delta) a(K,\lambda)} t \log((1 + \delta))  \ d\P(\lambda) \rightarrow \infty 
\end{align*}%
as $t\rightarrow \infty $. 
Thus by an appropriate choice of $t>D$ we can get arbitrarily large values for the almost sure $\Phi$-dimension of 
$\mu_\blambda$. 
Since 
\begin{equation*}
       t\mapsto t-\displaystyle\frac{\ee(\log \sum_{i}a(i,\lambda)^{t})}{\ee(\log a(K,\lambda))}
\end{equation*}
is a continuous function of $t$, we get our desired result by the
Intermediate value theorem in the case where $\P(\{a(1,\lambda)>a(K,\lambda)\})>0$.

To finish the proof of the theorem we now assume that $a(1,\lambda)=a(K,\lambda)$ almost surely 
(and thus all the scaling factors are the same almost surely). 
In this case, we set 
\begin{equation*}
          p(1)=q<1/K\ \text{ and }\ p(i)=\frac{1-q}{K-1}\text{ for }i=2,3,\ldots ,K.
\end{equation*}%
Since $a(i,\lambda)=a(j,\lambda)$ almost surely for all $i,j$, we have 
\begin{equation*}
        \frac{a(1,\lambda)^\theta}{p(1,\lambda)}\geq \frac{a(i,\lambda)^\theta}{p(i,\lambda)}\Leftrightarrow 
             \frac{1}{q}\geq \frac{1}{\frac{1-q}{K-1}}=\frac{K-1}{1-q}\Leftrightarrow q\leq 1/K
\end{equation*}
and thus since we choose $q < 1/K$ this means that $G(\theta)=\frac{\log q}{\ee(\log a(1,\lambda))}$ for all values of $\theta $. 
Hence $\overline{\dim}_{\Phi}\mu_\blambda=\log (q)/\ee(\log a(1,\lambda))$ almost surely. 
Clearly as $q\rightarrow 0$ this goes to $\infty $ and so again in this case we obtain any dimension in the range $[D,\infty)$.
\end{proof}

Next, we see that the analogous answer holds in the lower dimension case,
although the proof is necessarily different.

\begin{theorem} \label{thm:lower_dependent}
Let $d \in (0,D]$. 
There is a choice of probabilities $p(i,\lambda)$ so that almost surely $\udim_\Phi \mu_\blambda = d$ with $\mu_\blambda$ doubling.
\end{theorem}

\begin{proof}
We start by setting $p(K,\lambda)=q$ and $p(i,\lambda)=\frac{1-q}{K-1}$ for $i=1,2,\ldots ,K-1$ where $q\in [1/K,1)$. 
Then $\min_{i}a(i,\lambda)^\theta/p(i,\lambda)=a(K,\lambda)^\theta/q$ for any $\theta $ since $a(K,\lambda)\leq a(i,\lambda)$ and $q\geq 1/K$ implies $q\geq (1-q)/(K-1)$. 
Thus $Y'(\theta,\lambda)=\log (q)$ and $Z'(\theta,\lambda)=\log a(K,\lambda)$ for any $\theta \geq 0$ and so 
\begin{equation*}
   G'(\theta)=\frac{\log q}{\ee(\log a(K,\lambda))}.
\end{equation*}
This means that almost surely $\udim_{\Phi}\mu_\blambda=\dfrac{\log q}{\ee(\log a(K,\lambda))}$ which spans the entire
interval $(0,-\log K/\ee(\log a(K,\lambda))]$ as $q$ varies over $[1/K,1)$.

Notice that if $a(1,\lambda)=a(K,\lambda)$ almost surely, then the almost sure Hausdorff dimension of $F_\blambda$ is 
$D=\displaystyle\frac{-\log K}{\ee(\log a(K,\lambda))}$. 
That shows that in this case we see that we can obtain all possible lower dimensions.

So suppose that $\P(a(1,\lambda)>a(K,\lambda))>0$. 
We have seen that for any $d\in (0,\displaystyle\frac{-\log K}{\ee(\log a(K,\lambda))}]$ there is a choice
of probabilities which gives $\udim_{\Phi}\mu_\blambda=d$ almost surely. 
Thus we have only to show that this is also true for $d\in \displaystyle[\frac{-\log K}{\ee(\log a(K,\lambda))},D]$.

Let $0\leq t\leq D$ and set $p(j,\lambda)=\dfrac{a(j,\lambda)^t}{\sum_i a(i,\lambda)^t}$ for $j=1,2,\ldots ,K$. 
Notice that since $t\leq D$ we have $\ee(\log (\sum_{i}a(i,\lambda)^{t}))\geq 0$.

With this choice of $p(j,\lambda)$, we have 
\begin{equation*}
  Y'(\theta,\lambda)=\begin{cases}
          \log p(1,\lambda) & \text{ if }\theta \leq t, \\ 
          \log p(K,\lambda) & \text{ if }t<\theta 
     \end{cases}
     =
    \begin{cases}
         t\log a(1,\lambda)-\log \sum_{i} a(i,\lambda)^t & \text{ if }\theta \leq t, \\ 
         t\log a(K,\lambda)-\log \sum_{i} a(i,\lambda)^{t} & \text{ if }t<\theta ;
   \end{cases}
\end{equation*}
and
\begin{equation*}
    Z'(\theta,\lambda)=\begin{cases}
           \log a(1,\lambda) & \text{ if }\theta \leq t, \\ 
           \log a(K,\lambda) & \text{ if }t<\theta ;
    \end{cases}
\end{equation*}
and so 
\begin{equation*}
    G'(\theta)=  \begin{cases}
               t-\dfrac{\ee(\log \sum_{i} a(i,\lambda)^{t})}{\ee(\log (a(1,\lambda)))}, & \text{ if }\theta \leq t, \\[10pt]
               t-\dfrac{\ee(\log \sum_{i} a(i,\lambda)^{t})}{\ee(\log a(K,\lambda))}, & \text{ if }t<\theta.
    \end{cases}
\end{equation*}
Now $\ee(\log a(1,\lambda))<0$ and $\ee(\log a(K,\lambda))<0$, while ${\ee}(\log \sum_{i} a(i,\lambda )^{t})\geq 0$ (since $t\leq D$) and 
thus $G'(\theta)\geq t$ for any $\theta $. 
Thus  the smaller value of $G'$ is for $\theta > t$ and so
\begin{equation*}
    \udim_{\Phi}\mu_\blambda=t-\frac{\ee(\log \sum_{i} a(i,\lambda)^{t})}{\ee(\log a(K,\lambda))}=:f(t).
\end{equation*}
Note that $f(D)=D$ since $\ee(\log (\sum_{i} a(i,\lambda)^{D}))=0$ and $f(0)=-\frac{\log K}{\ee(\log a(K,\lambda))}$. 
Since $f$ is continuous as a function of $t$, the Intermediate Value Theorem guarantees that for any 
$d\in \displaystyle[\frac{-\log K}{\ee(\log a(K,\lambda))},D]$ there is some $t$ for which $f(t)=d$.

We note that $\mu_\blambda$ will be doubling in all the above cases since the probabilities are chosen to be strictly bounded away from zero.
\end{proof}

One immediate consequence of our results is a value for the almost sure $\Phi$-dimension of the random Moran set $F_\blambda$.

\begin{corollary} \label{cor:asdimF}
For any large dimension function $\Phi$ we have that 
\[
   \odim_\Phi F_\blambda = \udim_\Phi F_\blambda = D = \dim_H F_\blambda \ \text{ a.s.}
\]
where $D$ is the unique solution to
\[
      \ee(\log \sum_{i=1}^K a(i,\lambda)^x)=0.
\]
In addition,  $D=\max \udim_{\Phi}\mu_\blambda =\min \odim_{\Phi}\mu_\blambda$ 
where the maximum (or minimum) is taken over the associated random measures.
\end{corollary}

\begin{proof}
We always have $\odim_\Phi F_\blambda \le \odim_\Phi \mu_\blambda$ and since the random measure constructed in Theorem \ref{thm:lower_dependent} is doubling, we also have $\udim_\Phi \mu_\blambda \le \udim_\Phi F_\blambda$ for those measures.
Since we can arrange for $\odim_\Phi \mu_\blambda = D$ or $\udim_\Phi \mu_\blambda = D$, the conclusion follows.
\end{proof}

\section{The range of dimensions with independent probabilities}
\label{sec:independentrange}

We continue to use the notation described at the beginning of Section \ref{sec:range}, but now we consider the situation where the probability measure $\P$ on $\Lambda = \cA \times \cW$ is a 
product measure, $\P = \cQ \times \cR$. So we will fix $\cQ$ and allow any possible $\cR$.

In the case when the random Moran set, $F_\blambda$, is generated from the IFSs each of which has $K$ children, is equicontractive with common contraction factor $r(\lambda)$ and is strongly separated, then the situation turns out to be very simple, as we see first.
Almost surely $\dim _{H}F_{{\boldsymbol{\lambda }}}=D$ where $D$ satisfies 
${\ee}(\log K r^{D}(\lambda ))=0$. 
Solving this gives 
\begin{equation*}
   D=\frac{\log 1/K}{{\ee}(\log r(\lambda ))}.
\end{equation*}
Take the associated random measures with the fixed set of probabilities $\{p,p,...p,1-(K-1)p\}$, $0<p\leq 1/K$. 
As $p\leq 1-(K-1)p$, Theorem \ref{MainDimThm} gives 
\begin{equation*}
   \odim_{\Phi}\mu_\blambda=\frac{\log p}{\ee(\log r(\lambda))}\ \text{ and } \ 
    \udim_{\Phi}\mu_\blambda=\frac{\log (1-(K-1)p)}{\ee(\log r(\lambda))}\text{ a.s.}
\end{equation*}
By making suitable choices of $p\leq 1/K$ we obtain the full intervals $[D,\infty)$ 
and $(0,D]$ as in the dependent probability case.

Thus our focus will be on the general, non-equicontractive case.

\subsection{Do we obtain the set dimension as a measure dimension?}

Our study of the dimensions of random measures arising from independent
probabilities in the non-equicontractive case will begin by first seeing
that the answer to this question is, in general, no.

\begin{theorem}
\label{thm:general_gap_condition} 
Suppose there are $\epsilon >0$ and $t\in (0,1)$, sets $\Upsilon_{1}, \Upsilon_2 \subseteq \cA$ 
both of positive probability, and an index $i\in\{1,2,\ldots ,K\}$ so that 
\begin{equation*}
     \eta_1 := \sup_{\lambda_1 \in \Upsilon_1 \times \cW} \frac{a(i,\lambda_{1})^{D}}{\sum_{\ell =1}^Ka(\ell ,\lambda_{1})^{D}} <
                 \inf_{\lambda_2 \in \Upsilon_2 \times \cW} \frac{a(i,\lambda_{2})^{D}}{\sum_{\ell=1}^Ka(\ell ,\lambda_{2})^{D}} := \eta_2.
\end{equation*}
Then there are $\delta_1, \delta_2 >0$ so that for any choice of independent probabilities we have that 
$G(D)-D>\delta_1$ and $D - G'(D) > \delta_2$,
thus $\odim_{\Phi}\mu_\blambda>D+\delta_1$ and $\udim_\Phi \mu_\blambda < D - \delta_2$ almost surely for any large dimension function $\Phi$.
\end{theorem}

\begin{proof}
To simplify notation, let $\mfS(u)=\sum_{\ell=1}^Ka(\ell ,u)^{D}$ for a given $u \in \cA$. 
Choose $t$ and $\epsilon > 0$ so that $\eta_1 + \epsilon < t < \eta_2 - \epsilon$.
Next, choose $\lambda =(u,v)\in \Lambda = \cA \times \cW$ where $u_{1}\in \Upsilon_{1}$ and $u_{2}\in \Upsilon_{2}$ are fixed and consider $p(i,v)$. 
Either $p(i,v)\leq t$ or $p(i,v)>t$. 

If $p(i,v)\leq t$, then 
\begin{equation*}
     p(i,v)\leq t\leq \frac{a(i,u_{2})^{D}}{\mfS(u_{2})}-\epsilon 
\end{equation*}
and so with $m_{2}=m(D,(u_{2},v))$ we have 
\begin{equation*}
       \mfS(u_{2})\leq 
         \frac{a(i,u_{2})^{D}}{p(i,v)}-\epsilon \frac{\mfS(u_{2})}{p(i,v)}
      \leq \frac{a(m_{2},u_{2})^{D}}{p(m_{2},v)}-
        \epsilon \frac{\mfS(u_{2})}{p(i,v)}
                \leq \frac{a(m_{2},u_{2})^{D}}{p(m_{2}, v)}-\epsilon \mfS(u_{2}).
\end{equation*}
Thus 
\begin{equation*}
    (1+\epsilon)\mfS(u_{2})\leq \frac{a(m_{2},u_{2})^{D}}{p(m_{2},v)}\implies p(m_{2},v)
      \leq \frac{a(m_{2},u_{2})^{D}}{(1+\epsilon)\mfS(u_{2})}
\end{equation*}
and therefore 
\begin{equation}
     \log (p(m_2,v))\leq \log \left( \dfrac{a(m_2,u_2)^{D}}{\mfS(u_2)}\right) -\log (1+\epsilon).
\label{eq:gap_inequality2}
\end{equation}

On the other hand, if $p(i,v)>t$, then
\[
    p(i,v) > \frac{a(i,u_1)^D}{\mfS(u_1)} + \epsilon \implies \sum_{\ell \ne i} p(\ell,v) < \sum_{\ell \ne i}\left(\dfrac{a(\ell,u_1)^D}{\mfS(u_1)} - \frac{\epsilon}{K-1} \right)
\]
and thus there is some $j \in \{1,2,\ldots ,K\}$ with 
\begin{equation*}
    p(j,v)<\frac{a(j,u_{1})^{D}}{\mfS(u_{1})}-\frac{\epsilon}{K-1}.
\end{equation*}
Similar to the first case, with $m_1 = m(D,(u_1,v))$ this results in 
\begin{equation*}
      p(m_1,v)\leq \frac{a(m_1,u_{1})^{D}}{(1+\frac{\epsilon}{K-1})\mfS(u_{1})}
\end{equation*}
and so 
\begin{equation}
        \log (p(m_1,v))   \leq 
         \log \left( \dfrac{a(m_1,u_1)^D}{\mfS(u_1)}\right) -\log (1+\frac{\epsilon}{K-1}).  
   \label{eq:gap_inequality1}
\end{equation}

By the definition of $D$, $\ee_\lambda(\log(\mfS(\lambda))) = 0$ (for any $\lambda = (u,v) \in \cA \times \cW$), so we have
\begin{align*}
        G(D)-D& =\frac{\ee\left(Y(D,\lambda)-\log \left( \dfrac{a(m(D,\lambda),\lambda)^{D}}{\mfS(\lambda)}\right) \right)}{\ee(Z(D,\lambda))}
\\[8pt]
      & =\frac{1}{\ee(Z(D,\lambda))}\int_{\cA}\int_{\cW}\log (p(m(D,\lambda),\lambda))-\log \left( \dfrac{a(m(D,\lambda),\lambda)^{D}}{\mfS(\lambda)}\right) \ dv\ du
\\[8pt]
      & \geq \frac{-1}{\ee(Z(D,\lambda))}\P(\Upsilon_{2})\P(\{p(i,v)\leq t\})\log (1+\epsilon) \\[8pt]
      & +\frac{-1}{\ee(Z(D,\lambda))}\P(\Upsilon_{1})\P(\{p(i,v)>t\})\log (1+\frac{\epsilon}{K-1}) \\[8pt]
      & \geq \frac{1}{\left| \ee(Z(D,\lambda))\right|}\min \left\{ \P(\Upsilon_{2})\log (1+\epsilon),\P(\Upsilon_{1})\log (1+\frac{\epsilon}{K-1})\right\}.
\end{align*}
Notice that (\ref{eq:prob_inequality}) implies that 
\[
    \int_{A \setminus (\Upsilon_1 \cup \Upsilon_2)} \int_\cW \log (p(m(D,\lambda),\lambda))-\log \left( \dfrac{a(m(D,\lambda),\lambda)^{D}}{\mfS(\lambda)}\right) \ dv\ du \le 0,
\]
which, along with (\ref{eq:gap_inequality2}) and (\ref{eq:gap_inequality1}), give the first inequality in the development above.
Since the last line is always strictly positive, the existence of $\delta >0$ is established. 
From this we see that $\odim_{\Phi}\mu_\blambda=\sup_{\theta}G(\theta)\geq G(D)>D+\delta$.

The conclusions regarding the lower dimension and $G'$ are argued similarly.
\end{proof}

Theorem \ref{thm:general_gap_condition} is simpler in the case where there only finitely many different choices of sets of scaling factors.

\begin{corollary} \label{cor:gap_finite} 
Suppose that $\cA$ is finite. 

1. If there is an index $i\in \{1,...,K\}$ and indices $j\neq j'\in \cA$ such that 
\begin{equation*}
     \frac{a(i,j)^{D}}{\sum_{\ell =1}^K a(\ell ,j)^{D}}\neq \frac{a(i,j')^{D}}{\sum_{\ell =1}^K a(\ell ,j')^{D}},
\end{equation*}%
then there are $\delta_1, \delta_2 >0$ so that for any choice of independent
probabilities we have that $G(D)-D>\delta_1$ and $D - G'(D) > \delta_2$.
Thus $\odim_\Phi \mu_\blambda>D+\delta_1$ and $\udim_\Phi \mu_\blambda < D - \delta_2$
almost surely for any large dimension function $\Phi$.

\smallskip

2. If there is no such triple $\{i,j,j'\}$, then there is a choice
of independent probabilities for which $G(D)=D$, $G'(D) = D$ and $\udim_\Phi \mu_\blambda = \odim_\Phi \mu_\blambda=D$ almost surely for any large dimension
function $\Phi$.
\end{corollary}

\begin{proof}
For (2), the probabilities
\begin{equation*}
    p_{i}=\frac{a(i,j)^{D}}{\sum_{\ell =1}^K a(\ell ,j)^{D}}
\end{equation*}
 do not vary with $j$ and so $m(D,\lambda) = m'(D,\lambda)$ and thus $G(D) = D = G'(D)$.
Therefore by Theorem \ref{thm:G(t)=t} we have $\udim_\Phi \mu_\blambda = D = \odim_\Phi \mu_\blambda$.

For (1), let $i$, $j$, and $j'$ be such that
\begin{equation*}
    \eta =\left| \frac{a(i,j)^{D}}{\sum_{\ell =1}^K a(\ell ,j)^{D}}-
         \frac{a(i,j')^{D}}{\sum_{\ell=1}^K a(\ell ,j')^{D}}\right| > 0.
\end{equation*}
Without loss of generality we assume that 
\begin{equation*}
    \alpha :=\frac{a(i,j)^{D}}{\sum_{\ell =1}^K a(\ell ,j)^{D}}<\alpha'
             :=\frac{a(i,j')^{D}}{\sum_{\ell =1}^K a(\ell ,j')^{D}}.
\end{equation*}
Then we set $\epsilon =\eta /2$, $\Upsilon_{1}=\{j\}$, and $\Upsilon_{2}=\{j'\}$ and use Theorem \ref{thm:general_gap_condition}.
\end{proof}

Notice that in the situation of part (2) of the Corollary, since 
\[
  D = \udim_\Phi \mu_\lambda = \inf_\chi H(\chi) \le \sup_\chi H(\chi) = \odim_\Phi \mu_\blambda = D,
\]
we have $G(\theta) = G'(\theta) = D$ for all $\theta$ (because any value of either $G$ or $G'$ is $H(\chi)$ for some $\chi$).

\subsection{Attaining the full interval of possible measure dimensions}

In this subsection we will prove that it is possible to attain the full
range of possible measure dimensions, as in the dependent probabilities
case. Moreover, we are able to do this with the random measures arising from
single valid probability vectors.

It will be important to emphasize the dependence of the random measure and $G$ function on the probability weights.
Thus if we use the probability $\cR$ on the weight parameter space $\cW$, we denote the resulting random measure as $\mu_\cR$
and function as $G_\cR$.
In the special case that $\cR$ is supported on a single probability vector $p \in \cW$ we instead write $\mu_p$ and $G_p$.
In this case the choice of probability is deterministic and not random and thus the only 
randomness influencing the measure arises from the random choice of similarities.
We can naturally identify the set of valid single probabilities with a subset of 
$\mathbb{R}^{K}$ and give this set the inherited metric.

We first show that the almost sure $\Phi $-dimensions are continuous maps on
this metric space.

\begin{proposition} \label{prop:single_continuous}
The maps $p\rightarrow \sup_{\theta }G_{p}(\theta )=: M(p)$ and 
$p\rightarrow \inf_{\theta }G_{p}^{\prime }(\theta )=: M'(p)$ on the space of valid
single probabilities are continuous.
\end{proposition}

\begin{proof}
Fix $\varepsilon >0$ and valid single probability $q.$ Using the continuity
of the logarithm function, choose $\delta >0$ such that if $%
|p(j)-q(j)|<\delta $ for all $j\in \{1,...,K\}$, then $|\log p(j)-\log
q(j)|<\varepsilon $ for all $j$.

First, take $\theta _{q}$ such that $M(q)\leq G_{q}(\theta _{q})+\varepsilon 
$. As usual, let $m_{q}(\theta _{q},\lambda )$ be the minimal index $m_{q}$
such that 
\begin{equation*}
\frac{a(m_{q},\lambda )^{\theta }}{q(m_{q},\lambda )}=\max_{j}\frac{a(j,\lambda )^{\theta }}{q(j,\lambda )}.
\end{equation*}

Then
\begin{equation*}
      G_{q}(\theta _{q})=
    \frac{\mathbb{E(}\log q(m_{q},\lambda ))}{\mathbb{E(}\log a(m_{q},\lambda ))}
           \leq \frac{\mathbb{E(}\log p(m_{q},\lambda ))}{\mathbb{E(}\log a(m_{q},\lambda ))}+
 \frac{\varepsilon }{\left\vert \mathbb{E(}\log a(m_{q},\lambda ))\right\vert },
\end{equation*}
for any $p$ satisfying $\max_j |p(j) - q(j)| < \delta$.

Applying Lemma \ref{lem:H-function} with $\chi =m_{q}$ and recalling that 
$a(j,\lambda )\leq 1-A$ for all $j,\lambda$, we obtain
\begin{equation*}
   G_{q}(\theta _{q})\leq \sup_{\theta }G_{p}(\theta )+\frac{\varepsilon }{|\log 1-A|}.
\end{equation*}
Thus 
\begin{equation*}
M(q)\leq G_{q}(\theta _{q})+\varepsilon \leq M(p)+C\varepsilon .
\end{equation*}

But we can also reverse this argument. Again take any probability $p$ satisfying 
 $|p(j)-q(j)|<\delta $ for all $j$. 
Choose $\theta _{p}$ so that $M(p)\leq G_{p}(\theta _{p})+\varepsilon $ and use 
$m_{p}$ and Lemma \ref{lem:H-function} to deduce, as above, that 
\begin{equation*}
    M(p)\leq G_{p}(\theta _{p})+\varepsilon \leq \sup_{\theta }G_{q}(\theta ) +
     \frac{\varepsilon }{|\log 1-A|}=M(q)+C\varepsilon .
\end{equation*}
Thus 
\begin{equation*}
            |M(p)-M(q)|\text{ }<C\varepsilon ,
\end{equation*}
consequently, $M$ is continuous.

The argument for $M^{\prime }$ is symmetric.
\end{proof}

\begin{remark}
We remark that the same conclusion holds if we take $\mathcal{W}$ to be a
finite set, rather than a single element.
\end{remark}

\begin{corollary}
The maps $p\rightarrow \odim_{\Phi }\mu _{p}$ and 
$p\rightarrow \udim_{\Phi }\mu _{p}$ are continuous 
(where here by the $\Phi $-dimensions we mean the almost sure dimensions).
\end{corollary}

\begin{proof}
According to Theorem \ref{SupThm}, $M(p)=\odim_{\Phi }\mu _{p}$
a.s. and similarly for the lower dimension.
\end{proof}

Using this result, we will next prove that the minimal possible upper $\Phi $
measure dimension is attained by a measure arising from a single probability
vector.

\begin{lemma}
Let $d = \inf_{p } \odim_\Phi \mu_p$, where the infimum is over all valid single probabilities.  
Then for any distribution $\cR$ on $\cW$, we have $d \le \odim_\Phi \mu_\cR$.
\end{lemma}

\begin{proof}
Let $\psi < d$.   By part (i) of Theorem \ref{MainDimThm}, we must have that $\psi \le G_p(\psi)$ for any $p \in \cW$.
Thus for any fixed $v \in \cW$ we have that
\begin{equation*}
   \psi \int_\cA \left| \log a(m(\psi,u,v),u)\right| d\cQ(u) \leq
            \int_\cA \left| \log p(m(\psi,u,v),v)\right| d\cQ(u) 
\end{equation*}
and thus for any distribution $\cR$ on $\cW$ 
\begin{equation*}
    \psi \int_\cW \int_\cA | \log a(m(\psi,u,v),u) | d\cQ(u) d\cR(v) \le \int_\cW \int_\cA \left| \log p(\psi,m(\psi,u,v),v)\right| d\cQ(u) d\cR(v).
\end{equation*}
Hence $\psi \le G_\cR(\psi)$ which again by Theorem \ref{MainDimThm} means that $\odim_\Phi \mu_\cR \ge \psi$ for any $\cR$.
Since this is true for all $\psi < d$, we have $d \le \odim_\Phi \mu_\cR$ for any $\cR$ as claimed.
\end{proof}

\begin{corollary}
We have $\inf_p \odim_\Phi \mu_p = \inf_{\cR} \odim_\Phi \mu_\cR$  and also $\sup_p \udim_\Phi \mu_p = \sup_{\cR} \udim_\Phi \mu_\cR$ almost surely.
\end{corollary}

\begin{proof}
Of course $\inf_{\cR} \odim_\Phi \mu_\cR \le \inf_p \odim_\Phi \mu_p$ since point masses are special cases of general
distributions, hence the corollary is an immediate consequence of the previous proposition.
\end{proof}

It is also not hard to see that we can arrange to have arbitrarily large upper dimensions and arbitrarily small lower dimensions.

\begin{proposition}
\label{prop:approachzeroandinfinity}
Given $N$, there exists $\varepsilon >0$ such that if $p$ is any valid
single probability satisfying $\min p(j)<\varepsilon$ (or $\max p(j)>1-\varepsilon$), 
then $\odim_\Phi \mu_p > N$ (resp., $\udim_\Phi \mu_p<1/N$).
\end{proposition}

\begin{proof}
Take $C>0$ such that $1/| \ee (\log a(j,\lambda))| \geq C$ for all 
$j,\lambda $ and choose $\varepsilon >0$ so
that $C|\log (\varepsilon )|>N$.

Let $\min p(j)=p(j_{0})$ and suppose $p(j_{0})<\varepsilon$. 
Put $\chi=j_{0}$. 
In the notation of Lemma \ref{lem:H-function}, 
\begin{equation*}
   H(\chi )=\frac{|\ee \log p(j_{0})|}{|\ee (\log a(j_{0},u)|}\geq
    C|\log (\varepsilon )|>N
\end{equation*}%
and as that lemma tells us $H(\chi )\leq \sup G_{p}(\theta )=\odim_\Phi \mu _{p}$ 
a.s. we obtain the desired conclusion for the upper $\Phi$-dimension.

The lower dimension result is similar. 
\end{proof}

\begin{proposition} \label{prop:min_single}
There is a valid single probability $p_{0}$ such that $d=\odim_{\Phi}\mu_{p_{0}}$ a.s.
\end{proposition}

\begin{proof}
For each valid single probability $p$, let $M(p)=\sup_{\theta}G_{p}(\theta) = \odim_\Phi \mu_p$. 
Choose probabilities $p_{n} \in \cW$ such that $M(p_{n})\rightarrow \inf_p M(p)=d$. 
It follows from Proposition \ref{prop:approachzeroandinfinity} that there exists $\varepsilon >0$ such that 
$p_{n}\in [\varepsilon ,1-\varepsilon ]^K$ for large enough $n$,  consequently $(p_{n})$ has a
convergent subsequence (not renamed) with limit $p_{0}$ which is a valid probability.

However, as $M$ is continuous by Proposition \ref{prop:single_continuous}, $M(p_{n})$ converges to $M(p_{0}),$ whence $M(p_{0})=d$.
\end{proof}

The situation for the lower dimension is the same but the proof is more delicate.
For the upper dimension, Proposition \ref{prop:approachzeroandinfinity} shows that a small $p(i)$ leads to a large dimension, which allows for control of both uniform upper and lower bounds on $p(i)$.
However, just because $\max_i p(i) \le 1-\delta$ does not mean that $\min_i p(i) > 0$ and so the argument for the lower dimension is more involved.

\begin{proposition} \label{prop:max_single}
There exists a valid single probability $q$ such that 
$M^{\prime}(q)=\sup_{p}M^{\prime }(p)$ where the supremum is taken over all valid
single probabilities $p$.
\end{proposition}

\begin{proof}
Let $\Delta =\sup_{p}M^{\prime }(p)$ and choose valid single probabilities 
$p_{n}$ such that $\Delta _{n} := M^{\prime }(p_{n})\rightarrow \Delta$.
By passing to a subsequence (not renamed) we can assume $p_{n}\rightarrow p$.
Certainly, $p$ is a probability and if it is a valid probability then since $M^{\prime }$ 
is continuous we have $M^{\prime }(p)=\lim M^{\prime}(p_{n})=\Delta$ 
and we are done. 

So assume otherwise. 
We know from the Proposition \ref{prop:approachzeroandinfinity} that there exists $N_1$ and
$0<\delta <1$ such that for all $n\geq N_{1}$
and for all $j,$ $p_{n}(j)\leq 1-\delta $ and hence also $p(j)\leq 1-\delta$. 
Consequently, if $p$ is not valid it must be that the set $J=\{i:p(i)=0\}$ is not empty. 
Of course, as $p$ is a probability $J^{c}$ is also not empty.

Let $\eta =\min_{j\notin J}p(j)>0$. Now choose $\varepsilon >0$ so small
that 
\begin{equation*}
    K\varepsilon <\frac{\eta }{4}\text{ and }\frac{A^{\Delta }}{2\varepsilon }>\frac{4}{\eta }.
\end{equation*}

Pick $N_{2}\geq N_{1}$ such that for all $n\geq N_{2}$, 
$\left\vert p_{n}(i)-p(i)\right\vert <\varepsilon$ for all $i$ and 
$\left\vert A^{\Delta _{n}}-A^{\Delta }\right\vert <\varepsilon^2$. 
For $n\geq N_{2}$ we define 
\begin{equation*}
   q_{n}(i)=\left\{ \begin{array}{cc}
      p_{n}(i)+\varepsilon  & \text{if }i\in J \\ 
      p_{n}(i)-\frac{\varepsilon \left\vert J\right\vert }{K-\left\vert J\right\vert } & \text{if }i\notin J%
\end{array}%
\right. .
\end{equation*}
Clearly, if $i\in J$ then $\varepsilon \leq $ $q_{n}(i)\leq 2\varepsilon \leq \eta $ $\leq 1-\delta $, 
while if $i\notin J$ then $q_{n}(i)\leq p_{n}(i)\leq 1-\delta$.   
Furthermore, it is easy to see that if $i\notin J$, then  
\begin{equation*}
    q_{n}(i)>p(i)-\varepsilon -\frac{\varepsilon \left\vert J\right\vert }{K-\left\vert J\right\vert }
          >\frac{\eta }{2}\geq \varepsilon.
\end{equation*}
Since also $\sum_{i}q_{n}(i)=1,$ each  $q_{n}$ is a valid probability and
even satisfies  $q_{n}(j)\in [\varepsilon ,1-\delta ]$. 

We have 
\begin{equation*}
    \frac{A^{\Delta _{n}}}{2\varepsilon }>\frac{A^{\Delta }-\varepsilon ^{2}}{2\varepsilon }>
      \frac{4}{\eta }-\frac{\varepsilon }{2}>\frac{2}{\eta }.
\end{equation*}
Thus if $n\geq N_{2}$ and $i\in J$, then
\begin{equation*}
    \frac{a^{\Delta _{n}}(i,\lambda )}{q_{n}(i)}=
        \frac{a^{\Delta _{n}}(i,\lambda)}{p_{n}(i)+\varepsilon }
      \geq \frac{a^{\Delta _{n}}(i,\lambda )}{2\varepsilon }
         \geq \frac{A^{\Delta _{n}}}{2\varepsilon }>\frac{2}{\eta },
\end{equation*}
while if $i\notin J$, then 
\begin{equation*}
    \frac{a^{\Delta _{n}}(i,\lambda )}{q_{n}(i)}\leq 
        \frac{a^{\Delta_{n}}(i,\lambda )}{^{\eta /2}} < \frac{2}{\eta }.
\end{equation*}
This shows that $m_{n}^{\prime } := m_{q_{n}}^{\prime }(\Delta_{n},\lambda )\notin J$ 
for any $\lambda$. Consequently, 
\begin{eqnarray*}
     \Delta _{n} &=&\underline{\dim }_{\Phi }\mu _{p_{n}}\leq
     H_{p_{n}}(m_{n}^{\prime })=
     \frac{\ee(\log p_{n}(m_{n}^{\prime }))}{\ee(\log a(m_{n}^{\prime },\lambda ))}\\
    &=&\frac{\ee( \log\left( q_{n}(m_{n}^{\prime })+\frac{\varepsilon |J|}{K- |J| }\right))}
           {\ee(\log a(m_{n}',\lambda ))}
        <\frac{\ee(\log q_{n}(m_{n}^{\prime }))}{\ee(\log a(m_n',\lambda ))}
                =G_{q_n}'(\Delta_n).
\end{eqnarray*}
Since $G_{q_{n}}^{\prime }(\Delta _{n})>\Delta _{n}$, Theorem \ref{MainDimThm}
implies that  $\underline{\dim }_{\Phi }\mu _{q_{n}}\geq \Delta _{n} \to \Delta$. 
But, of course, $M^{\prime }(q_{n})= \udim_\Phi \mu_{q_n} \leq \sup_{p} \udim_\Phi \mu_p=\Delta$.

As $q_{n}(j)\in [\varepsilon ,1-\delta]$ for all $j$, a subsequence
(not renamed) converges to a valid probability $q$ and by continuity, 
$M'(q_n) \to M'(q)$. 
Thus  $M'(q)=\Delta$ for a single valid probability $q$, as we desired to prove.
\end{proof}

\begin{theorem} 
\label{thm:singlefull}
Let $d=\inf_\cR \odim_\Phi \mu _\cR $ and 
$d^{\prime }=\sup_\cR \udim_\Phi \mu _\cR$. 
For any $t\in \lbrack d,\infty )$ (or $t^{\prime }\in (0,d^{\prime }]$) there is some valid single probability $p$ such that $\odim_{\Phi }\mu _{p}=t$ a.s. (resp., 
$\udim_{\Phi }\mu_{p}=t^{\prime }$ a.s.).
\end{theorem}

\begin{proof}
We have already seen that $d$ and $d^{\prime }$ can be attained in this fashion. 
Since Prop. \ref{prop:approachzeroandinfinity} shows that there are measures 
$\mu _{p}$ with arbitrarily large upper $\Phi $-dimension and others with 
arbitrarily small lower $\Phi $-dimension, and the functions $M(p)$ and $M^{\prime }(p)$ are continuous, the result follows immediately
from an application of the Intermediate value theorem.
\end{proof}

\begin{remark}
We remark that whether $0$ or $\infty $ can be attained depends on the
underlying probability space.
\end{remark}

\section{The minimal measure dimension when each IFS has two similarities}
\label{sec:algorithm}

In this section we will continue our study of the minimal attainable
dimension of the random Moran measures arising from independent
probabilities, specializing to the case of finitely many IFSs
each of which consists of two similarities. 
We will describe a simple algorithm for determining the minimum possible dimension which we have already seen is
equal to $\min_{p}M(p)$ (see Proposition \ref{prop:min_single}) where 
\begin{equation*}
    M(p)=\sup_{\theta }G_{p}(\theta )
\end{equation*}
and the minimum is taken over all valid single probability vectors $p$ (in this case meaning,
the probability weights $(p,1-p)$).
Thus we assume that $\cA = \{ 1, 2, \ldots, L \}$, $K = 2$, and $\cW$ is a single point.
We set $\pi_i = \cQ(\{ i \})$ to be the likelihood of choosing the $i$'th IFS. 

Using these assumptions, all of the possible values of $Z_p(\theta)$ and $Y_p(\theta)$, and thus $G_p(\theta)$, come from a finite
set of \emph{log-ratios} each of the form
\[
    \frac{ \sum_i \pi_i \log( q_i)}{\sum_i \pi_i \log( r_i)}
\]
where each $q_i$ is choice of either $p$ or $1-p$ and $r_i$ a choice of left or right scaling; there are $2^L$ such possible expressions.
Our analysis involves a careful consideration of which of these gives the dimension as we change the value of $p$.
In particular, it turns out that we only have to consider at most $L+1$ such ratios, which we will denote by $f_j$ for $j=0,1,\ldots, L$ (see below).
Figure \ref{fig:Jfunctions} gives some indication of how these functions typically look.

\subsection{Further notation and some elementary facts}

In this situation, we can assume that the left and right contraction factors of the $i$'th IFS are of the form  
$a^{\alpha _{i}}$, $a^{\beta _{i}}$ for $i=1,...,L$ and some $a > 0$ where, without loss of
generality, we assume $(\beta _{i}-\alpha _{i})$ is decreasing. 
 We will suppose that 
\begin{eqnarray*}
   \beta _{i}-\alpha _{i} &\geq &0\text{ for }i\in \{1,...,N\}\text{ and} \\
    \beta _{i}-\alpha _{i} &<&0\text{ for }i\in \{N+1,...,L\}
\end{eqnarray*}%
and assume that there is some index $i$ with $\beta _{i}-\alpha _{i}\geq 0$.
There is no loss of generality in making this assumption since if $\beta_i - \alpha _i<0$ for all $i$ we 
instead consider the random set generated by the IFSs with left contractions $a^{\beta _{i}}$ and right contractions 
$a^{\alpha _{i}}$. 
The random measure supported on this new random set and associated with the single probability 
$(1-p,p)$ mirrors the measure supported on the original random set arising from the probability $(p,1-p)$.
Hence the second random set has the same minimum measure dimension as the
original.

Our ordering assumption ensures that for all $\theta \geq 0$, 
\begin{equation*}
\frac{1}{1+a^{(\beta _{i+1}-\alpha _{i+1})\theta }}=\frac{a^{\alpha
_{i+1}\theta }}{a^{\alpha _{i+1}\theta }+a^{\beta _{i+1}\theta }}\leq \frac{%
a^{\alpha _{i}\theta }}{a^{\alpha _{i}\theta }+a^{\beta _{i}\theta }}=\frac{1%
}{1+a^{(\beta _{i}-\alpha _{i})\theta }}
\end{equation*}%
and that for $\theta \geq 0$ 
\begin{eqnarray*}
\frac{1}{1+a^{(\beta _{i}-\alpha _{i})\theta }} &\in &(0,1/2]\text{ for }%
i\geq N+1\\ &\text{ and}& \\
\frac{1}{1+a^{(\beta _{i}-\alpha _{i})\theta }} &\in &[1/2,1)\text{ for }%
i\leq N.
\end{eqnarray*}

For  $j=0,1,...,L$  let
\begin{equation*}
 S_j = \sum_{i=1}^j \pi_i \quad \mbox{ and } \quad T_{j}=\sum_{i=1}^{j}\alpha _{i}\pi_i+\sum_{i=j+1}^{L}\beta _{i}\pi_i,
\end{equation*}%
with suitable interpretations when $j = 0$ or $L$.
Given all of this, one can easily see that for each $\theta \geq 0$, 
\begin{equation*}
G_{p}(\theta )=\left\{ 
    \begin{array}{cc}
     f_{0}(p) & \text{if }p>\dfrac{1}{1+a^{(\beta _{1}-\alpha _{1})\theta }} \\[10pt]
    f_{j}(p) & \text{if }\dfrac{1}{1+a^{(\beta _{j+1}-\alpha _{j+1})\theta }}<p\leq \dfrac{1}{1+a^{(\beta _{j}-\alpha _{j})\theta }},\text{ }j=1,...,L-1
\\[10pt]
    f_{L}(p) & \text{if }p\leq \dfrac{1}{1+a^{(\beta _{L}-\alpha _{L})\theta }}
\end{array}
    \right. 
\end{equation*}
where 
\begin{equation*}
   f_{j}(p)=\frac{S_j\log p+(1-S_j) \log (1-p)}{ T_{j}\log a}\quad \text{ for } j=0,...,L. 
\end{equation*}
If we suppose, in addition, that all $\pi_i=1/L$ (that is, the
IFS are chosen with equal likelihood), then the functions $f_{j}$ simplify
to 
\begin{equation*}
    f_{j}(p)= \frac{j\log p+(L-j)\log (1-p)}{(\sum_{i=1}^{j}\alpha_{i}+\sum_{i=j+1}^{L}\beta _{i})\log a}.
\end{equation*}

\smallskip

Figure \ref{fig:Jfunctions} illustrates two different models, their collections of $f_j$s, 
along with a zoom on the range of $p$ values which give the minimal measure dimension.
Both models use an equally likely choice among the IFSs.
For the top two images, the model has $L = 10$ distinct sets of scaling factors with $a = 1/2$
and the $\alpha_i$ and $\beta_i$, in order, given as:

\smallskip

\begin{tabular}{l@{\hskip 0.2 cm}cccccccccc}
   $\alpha$:  & 1.1, & 1.3, & 1.5, & 1.8, & 1.7, & 1.9, & 1.6, & 2.9, & 5, & 7\\
   $\beta$:  & 10.1, & 7.0, & 4.6, & 4.2, & 2.8, & 2.2, & 1.7, & 1.6, & 2.8, & 3.
\end{tabular}
\smallskip

\noindent Note that $N = 7$ in this case.
From numerical computations, for this model we have $D \approx 0.3476$ while the 
minimal possible measure dimension $\min_p M(p) \approx 0.5360$ and so there is a ``gap'' as
proven in Cor. \ref{cor:gap_finite}.

For the bottom two images, the model has $L=3$ distinct sets of scaling factors with $a=1/3$
and the $\alpha_i$ and $\beta_i$, in order, given as:
\[
   \alpha: \ 1.1, 1.1, 1 \quad \mbox{ and } \quad \beta: \ 4.1, 3.1, 2.
\]
Notice that $N = 3$ in this second case.
In this model, numerical computations give $D \approx 0.3398$ while $\min_p M(p) \approx 0.4138$, so again 
there is a gap.

At any specific choice of $p$, the function $f_j(p)$ with the largest value will be the one which gives
the measure dimension for that choice of $p$.
Thus what we need to do is track  how this ``top'' $f_j$ changes with $p$ and find the 
$p$ which minimizes $\max_j f_j(p)$.
Clearly the places where the ``top'' $f_j$ changes are going to be important, as are the local minima for each $f_j$.

We encourage the reader to refer to Figure \ref{fig:Jfunctions} when reading the results and the algorithm in this section.

\begin{figure}[tbp]
\includegraphics[width=1.6in, height=1.6 in]{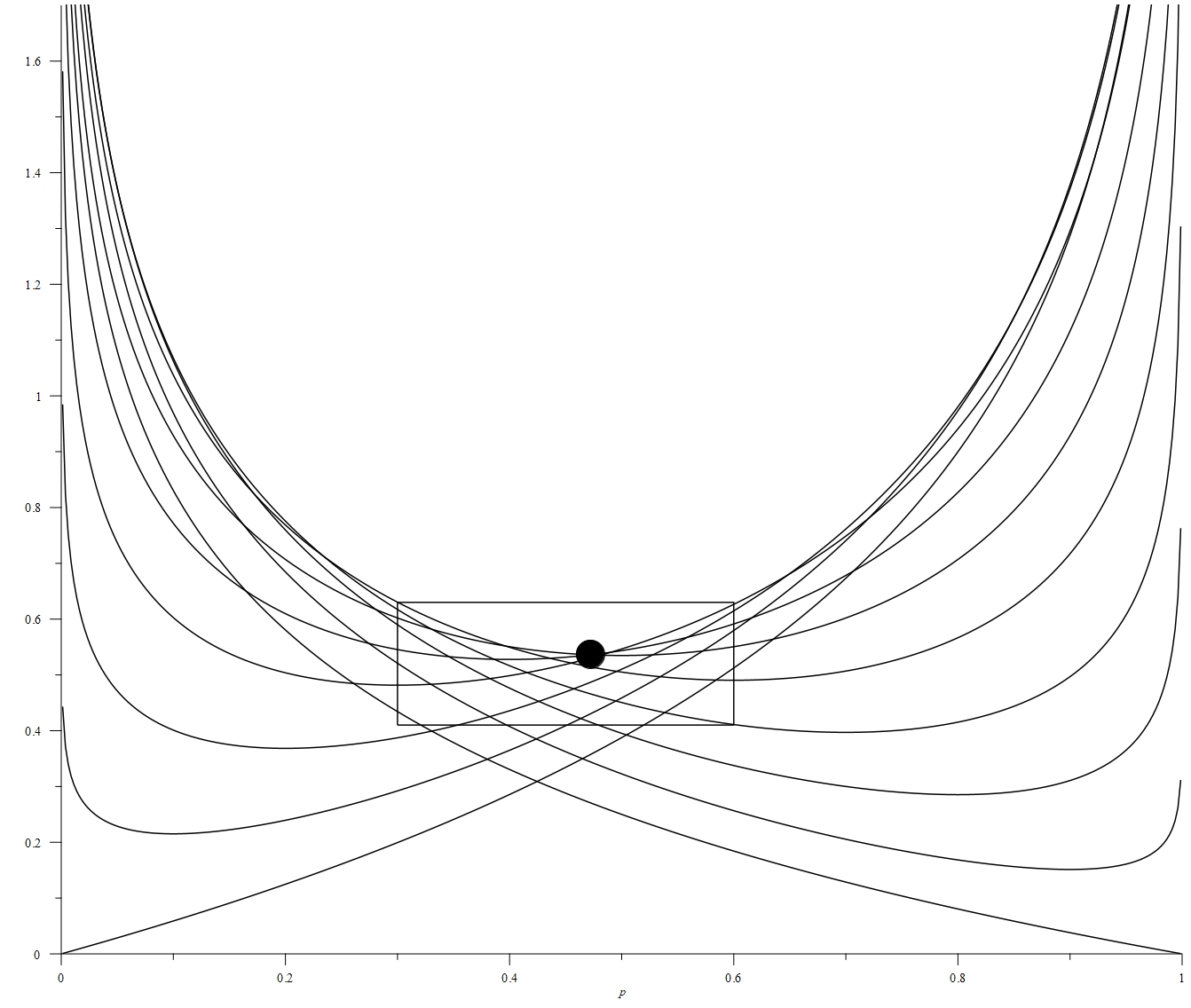}
\includegraphics[width=1.6in, height=1.6in]{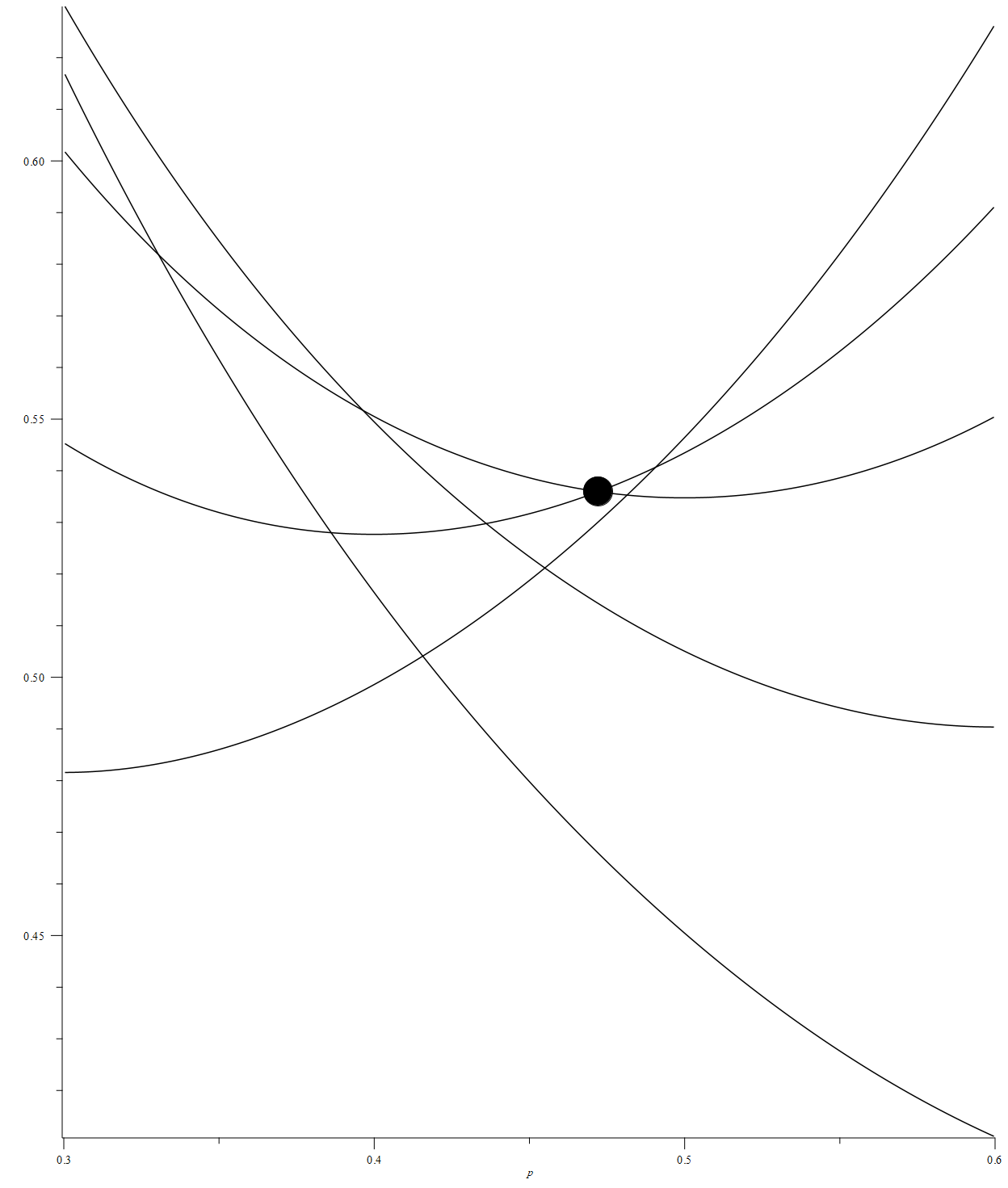}
\par
\includegraphics[width=1.6in, height = 1.2in]{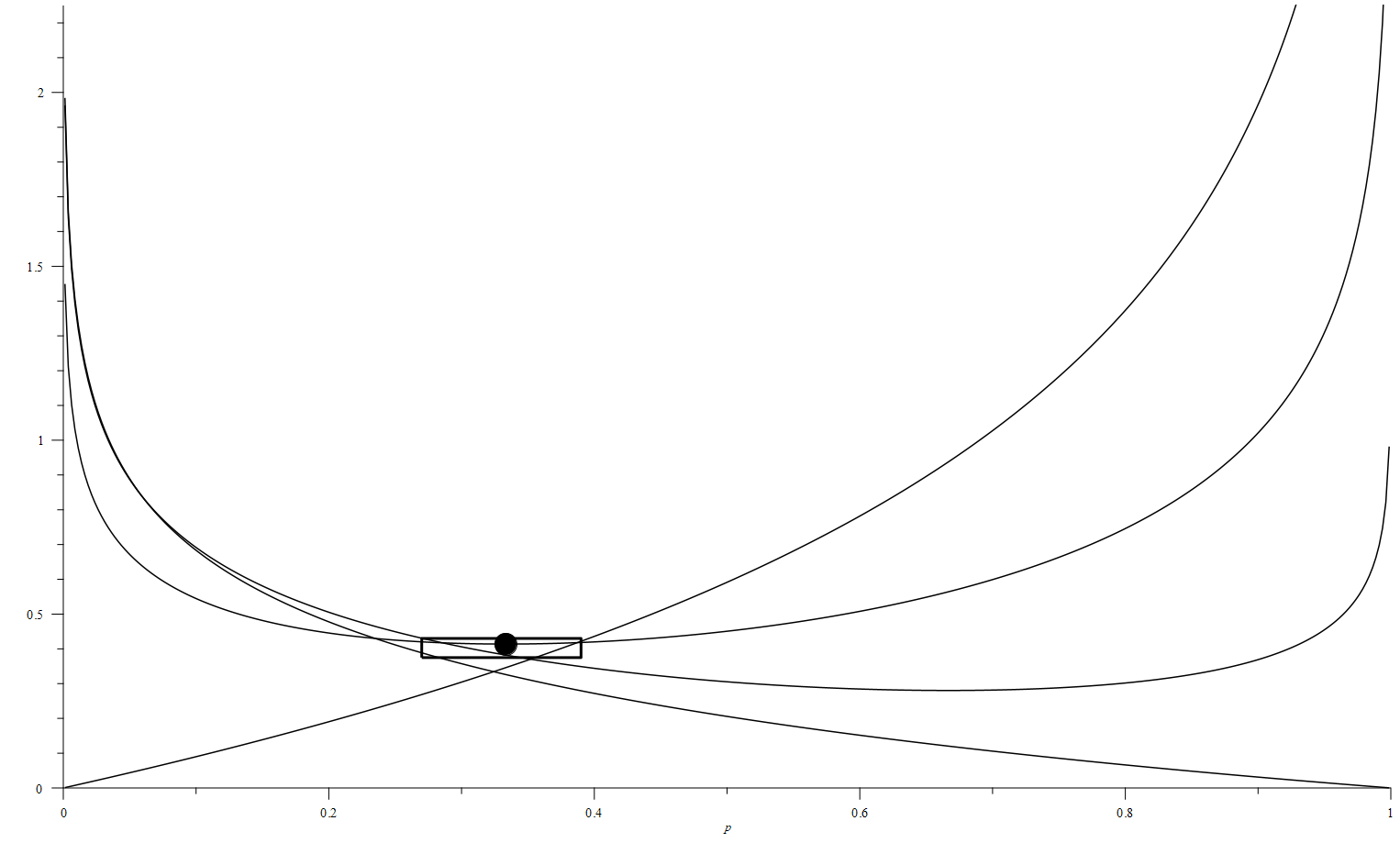}
\includegraphics[width=1.6in,height=1.6 in]{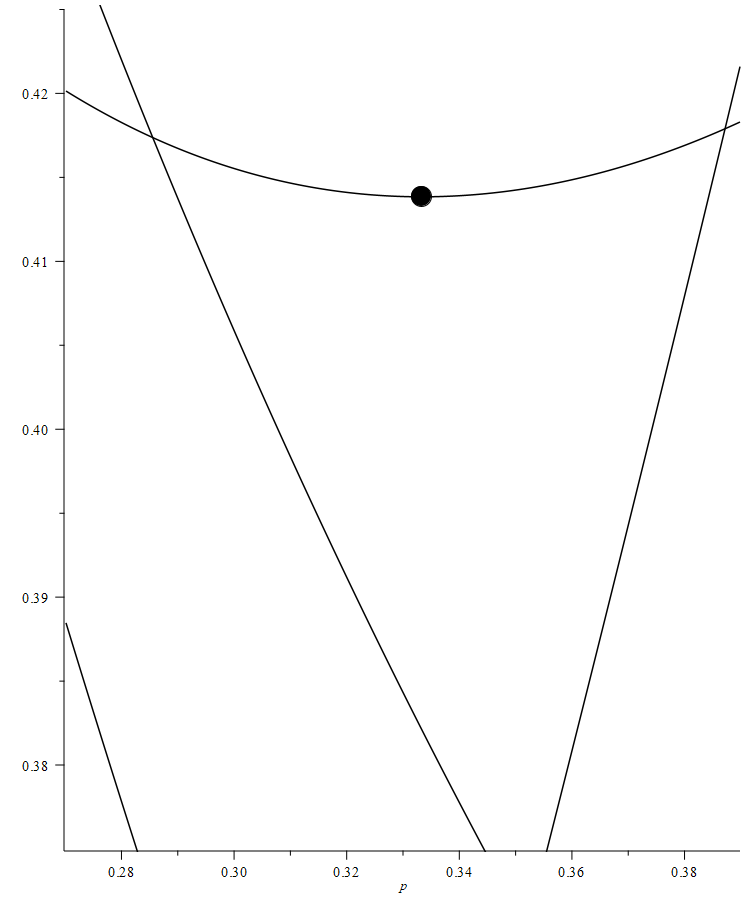}
\caption{Two examples of plots of the $f_j$s, each with a zoom around the point, marked with a dot, of minimal dimension.}
\label{fig:Jfunctions}
\end{figure}

We will further assume that the sequence $(\beta _{i}-\alpha _{i})_{i=1}^{L}$ is strictly decreasing. 
This assumption implies, in particular, that $\beta_{j}-\alpha _{j}=0$ can only hold for $j=N$. 
This is mainly a notational convenience as we will briefly explain at the end of the section.

We refer to the inequality 
\begin{equation}
   1/(1+a^{(\beta _{j+1}-\alpha _{j+1})\theta })<p\leq 1/(1+a^{(\beta
_{j}-\alpha _{j})\theta })  \tag{Condition $j$}
\end{equation}
as ``Condition $j$'' (where we understand this suitably in the case of $j=0$ or $L$). 
For any $p<1/2$, condition $j$ is satisfied for some $\theta \geq 0$ when 
$j\geq N$, while if $p>1/2$, then condition $j$ can be satisfied for any $j\leq N-1$, 
as well as for $j=N$ if $\alpha _{N}\neq \beta _{N}$.

Put 
\begin{equation*}
   \theta _{j}(p):=\frac{\log\left( (1-p)/p \right)}{(\beta _{j}-\alpha _{j})\log a}
\end{equation*}%
with $\theta _{N}=\infty $ if $\alpha _{N}=\beta _{N}$. 
With this notation, $p<1/2$ satisfies 
\begin{eqnarray*}
    \text{Condition $j$  if }\theta  &\in &(\theta _{j+1},\theta _{j}]\text{ for }N+1\leq j<L, \\
    \text{Condition $L$ if }\theta  &\leq &\theta _{L}\text{ and Condition $N$ if }\theta >\theta _{N+1},
\end{eqnarray*}%
and for no other $j$'s. When $p>1/2$, $p$ satisfies 
\begin{eqnarray*}
   \text{Condition $j$ if }\theta  &\in &[\theta _{j},\theta _{j+1})\text{ for } 1\leq j\leq N-1, \\
    \text{Condition $0$ if }\theta  &<&\theta _{1}\text{ and Condition $N$ if }\theta \geq \theta _{N},
\end{eqnarray*}
and for no other $j$'s. 

Condition $N$ holds for $p=1/2$ for all $\theta >0$ and is the only
condition which may be valid for both choices of $p<1/2$ and $p>1/2$. 

Since $T_{j+1}=T_{j}+(\alpha _{j+1}-\beta _{j+1})\pi_{j+1}$, we
clearly have 
\begin{equation*}
  T_{j+1}\leq T_{j}\text{ if }j\leq N-1\text{ and }T_{j+1}\geq T_{j}\text{ if } j\geq N,
\end{equation*}
where these inequalities are strict unless $j=N-1$ and $\beta _{N}-\alpha_{N}=0$

Here are other useful facts that are easy to prove.

\begin{lemma}
\label{calculus}
(a) The function $p\rightarrow f_{j}(p)$ strictly decreases for $p\in (0, S_j]$ and strictly increases for $p\in [ S_j, 1)$.

(b) If either $f_{j}$ is increasing and $j\geq N$ or $f_{j}$ is decreasing
and $j\leq N-1$, then $f_{j}-f_{j+1}$ is increasing.

(c) $f_{i}(p)=f_{j}(p)$ has at most two solutions for any $i\neq j$.
\end{lemma}

\begin{proof}
(a) follows immediately from the fact that 
\begin{equation*}
  f_{j}^{\prime }(p)=\frac{S_j - p}{p(1-p)T_{j}\log a}.
\end{equation*}

(b) Likewise, 
\begin{equation*}
  f_{j}^{\prime }(p)-f_{j+1}^{\prime }(p)=\frac{1}{p(1-p)\log a}\left( (S_j - p)(\frac{1}{T_{j}}-\frac{1}{T_{j+1}})-\frac{\pi_{j+1}}{T_{j+1}}\right) .
\end{equation*}

Assume $f_{j}$ is increasing and $j\geq N$. Then $S_j - p \leq 0$ and $1/T_{j+1}\leq 1/T_{j}$, 
which implies $f_{j}^{\prime }(p)-f_{j+1}^{\prime}(p)>0$.

The other case is similar.

(c) This follows from the fact that $(f_{i}-f_{j})^{\prime }$ has at most
one root.
\end{proof}

\begin{lemma}
If $\alpha _{N}\neq \beta _{N}$, then $M(1/2)=f_{N}(1/2)>f_{j}(1/2)$ for all 
$j\neq N$. Otherwise, $M(1/2)=f_{N}(1/2)=f_{N+1}(1/2)>f_{j}(1/2)$ for all $%
j\neq N,N+1$.
\end{lemma}

\begin{proof}
For all $j$, 
\begin{equation*}
f_{j}(1/2)=\frac{\log 1/2}{T_{j}\log a}.
\end{equation*}%
The ordering of the $T_{j}$'s implies that $f_{j}(1/2)<f_{j+1}(1/2)$ if $j<N$%
, $f_{j}(1/2)>f_{j+1}(1/2)$ if $j>N$ and $f_{N}(1/2)\geq f_{N+1}(1/2)$ with
equality only when $\alpha _{N}=\beta _{N}$. Further, $G_{1/2}(\theta
)=f_{N}(1/2)$ for $\theta >0$.
\end{proof}

\subsection{Transitioning}

We will say that ``$M$ left transitions from $f_{i}$ to $f_{j}$ at $b$, for $j\neq i$'', 
if there is some $\varepsilon >0$ such that

(i) for each $p\in (b,b+\varepsilon )$, $p$ satisfies condition $i$ and $f_{i}(p)>f_{l}(p)$ for all $l\neq i;$ and

(ii) for each $p\in (b-\varepsilon ,b)$, $p$ satisfies condition $j$  and $f_{j}(p)>f_{l}(p)$ for all $l\neq j$.

\smallskip

We define right transitioning analogously.
The idea behind this definition is that the ``top'' $f_i$ switches at $b$, either when moving to the right or left.

We remark that if $M$ transitions from $f_{i}$ to $f_{j}$ at $b<1/2$ we must
have $i,j\geq N$ since conditions $i,j$ can only be satisfied for $p$ near $b
$ for such indices. Similarly, if the transition occurs at $b>1/2$, $i,j$ $%
\leq N$.

Here is the key technical lemma.

\begin{lemma}
\label{transition}(a) If $M$ left (right) transitions from $f_{i}$ to $f_{j}$
at $b<1/2$ ( $b>1/2$), then $j=i\pm 1$.

(b) If $M$ left (right) transitions from $f_{N}$ to $f_{j}$ at $b<1/2$ ( $%
b>1/2$), then $j=N+1$ ( $j=N-1$ resp.).

(c) Suppose $b<1/2$, $M$ left transitions from $f_{i}$ to $f_{j}$ at $b$ and 
$f_{i}$ is increasing at $b$. Then $j=i+1$. Similarly, if $b>1/2$, $M$ right
transitions from $f_{i}$ to $f_{j}$ at $b$ and $f_{i}$ is decreasing at $b$,
then $j=i-1$.
\end{lemma}

\begin{proof}
(a) We will prove the left transitioning case as the right case is symmetric. 
For $p\in (b,b+\varepsilon )$ we have $M(p)=f_{i}(p)$ and $i$ is the unique index having this property. 
Thus by Theorem \ref{thm:G(t)=t} we know that $f_{i}(p)=G_{p}(f_{i}(p))$. 
Similarly, for $p\in (b-\varepsilon ,b)$ we have $M(p)=f_{j}(p)$ and $j$ is the unique index having this property, 
so $f_{j}(p)=G_{p}(f_{j}(p))$. 
By continuity, $f_{i}(p)\rightarrow f_{i}(b)$ as $p\downarrow b$, while $f_{j}(p)\rightarrow f_{j}(b)$ as $p\uparrow b$. 
Since $M$ is continuous by Proposition \ref{prop:finite_GContinuous}, it follows
that $f_{i}(b)=f_{j}(b)$. 

First, suppose $i,j\neq N$. Since $G_{p}(f_{i}(p))=f_{i}(p)\neq f_{\ell }(p)$
for $\ell \neq i$ when $p\in (b,b+\varepsilon )$, it follows that $%
f_{i}(p)\in (\theta _{i+1}(p),\theta _{i}(p)]$ for such $p$. Similarly, $%
f_{j}(p)\in (\theta _{j+1}(p),\theta _{j}(p)]$ for $p\in (b-\varepsilon ,b)$.
Because the functions $p\rightarrow \theta _{l}(p)$ are also continuous, it
follows that $f_{i}(b)\in \lbrack \theta _{i+1}(b),\theta _{i}(b)]$ and $%
f_{j}(b)\in \lbrack \theta _{j+1}(b),\theta _{j}(b)]$. Hence either $i=j+1$
or $j=i+1$ as claimed.

The argument is similar if $i$ or $j$ $=N$.

(b) This follows immediately from the fact that in the left transition case
we assume $b<1/2$, so $j\geq N$ and hence $j\neq N-1$. Similarly $j\neq N+1$
in the right transition case.

(c) Consider the left transition case. As $b<1/2$,  $i$ $\geq N$, and since
we have already done the case $i=N$ in (b) we can assume that $i\geq N+1$.
As $f_{i}$ is increasing on $[b,1)$, so is $f_{i-1}$ and thus Lemma \ref%
{calculus}(b) shows that $f_{i-1}-f_{i}$ is increasing on $[b,1)$. If $f_{i}$
left transitioned to $f_{i-1}$, then $f_{i}(b)=f_{i-1}(b)$ and thus the
increasing property would yield $f_{i-1}(b+\varepsilon )>f_{i}(b+\varepsilon
)$ contradicting the definition of left transitioning.

The right transition case is similar.
\end{proof}

\subsection{The Algorithm}

Case 1. Suppose $\alpha _{N}\neq \beta _{N}$. We have already shown that $%
M(1/2)$ $=f_{N}(1/2)>f_{j}(1/2)$ for all $j\neq N$, hence $\min_{p}M(p)\leq
f_{N}(1/2)$.

Sub-case 1a: $f_{N}$ is increasing at $1/2$. Then for $p>1/2$, $M(p)\geq
f_{N}(p)>f_{N}(1/2)$, so to minimize $M$ we must consider $p\leq 1/2$. By
continuity $f_{N}(p)$ will continue to be the unique maximum as we decrease $%
p$ until the largest $b<1/2$ where $f_{N}$ left transitions to $f_{j}$,
i.e., $M(p)=f_{N}(p)$ for all $p\in \lbrack b,1/2]$.

If $b\leq S_{N}$, then $M(S_{N})=f_{N}(S_{N})\leq f_{N}(p)\leq M(p)$ for all 
$p\geq S_{N}$. Further, $f_{N}$ is decreasing on $(0,b]$, which ensures that 
$M(p)\geq f_{N}(p)\geq f_{N}(b)\geq f_{N}(S_{N})$ for all $p\leq b$.
Therefore $\min_{p}M(p)=f_{N}(S_{N})$.

If, instead, $b>S_{N}$, then $f_{N}$ is increasing at $b$ and therefore $M(b)=f_{N}(b)\leq f_{N}(p)\leq M(p)$ for all $p>b$. 
Lemma \ref{transition}(c) tells us that $j=N+1$. If $f_{N+1}$ is decreasing at $b$, then 
$M(p)\geq f_{N+1}(p) \geq f_{N+1}(b)=f_{N}(b)=M(b)$ for all $p\leq b$ and thus $\min_{p}M(p)=f_{N}(b)$.

Otherwise, $f_{N+1}$ is increasing at $b$. Then $M(p)\geq
f_{N+1}(p)>f_{N+1}(b)=M(b)$ for all $p>b$, so to minimize $M$ we must
consider $p\leq b$ and we basically repeat the argument. We have that $%
f_{N+1}(p)$ is the unique maximum for $p\in (b_{1},b)$ where $b_{1}<$ $b$ is
maximal with the property that $f_{N+1}$ left transitions to some $f_{j}$ at 
$b_{1}$. If $b_{1}\leq S_{N+1}$, then $f_{N+1}$ is decreasing at $b_{1}$ and
hence $\min_{p}M(p)=f_{N+1}(S_{N+1})$. If $f_{N+1}$ is increasing at $b_{1}$
then Lemma \ref{transition} tells us that $j=N+2$. If $f_{N+2}$ is
decreasing at $b_{1}$, then similar reasoning to that given above shows that 
$\min_{p}M(p)=f_{N+1}(b_{1})$.

Otherwise, $f_{N+2}$ is increasing at $b_{1}$ and we repeat the argument.

This argument must terminate in a finite number of steps producing a choice
of $p$ which minimizes $M$. Furthermore, it is clear from the argument that
this is the unique choice of a (single) probability minimizing $M$.

Sub-case 1b: $f_{N}$ is decreasing at $1/2$. The argument is symmetric. We
have $M(p)\geq f_{N}(p)>f_{N}(1/2)$ for all $p<1/2$, so to minimize $M$ we
must consider $p\geq 1/2$. We find the smallest $b>1/2$ where $f_{N}$ right
transitions to $f_{j}$ at $b$. If $b\geq S_{N}$, then similar reasoning to
above shows that $M(S_{N})=f_{N}(S_{N})=\min_{p}M(p)$.

If $b<S_{N}$, then $f_{N}$ is decreasing at $b$ and $j=N-1$. If $f_{N-1}$ is
increasing at $b$, then $\min_{p}M(p)=f_{N}(b)$, while if $f_{N-1}$ is
decreasing at $b$ we repeat the argument. Again, it must terminate in a
finite number of steps with the choice of $p$ which minimizes $M$.

This completes the algorithm when $\alpha _{N}\neq \beta _{N}$.

Case 2. Now suppose $\alpha _{N}=\beta _{N}$, so that $M(1/2)=f_{N}(1/2)=f_{N+1}(1/2)>f_{j}(1/2)$ for all $j\neq N,N+1$.

If $f_{N}$ is decreasing at $p=1/2$, then $M(p)>f_{N}(1/2)$ for $p<1/2$, so
to minimize $M$ we must consider $p>1/2$. Since condition $N+1$ does not
hold for $p>1/2$, $f_{N}$ will be the unique maximum on some interval $(1/2,1/2+\varepsilon )$ 
and thus we simply repeat the argument of case 1b.

If $f_{N}$ is increasing at $p=1/2$, then we must consider $p<1/2$ since $%
M(p)>f_{N}(1/2)$ for $p>1/2$. If $f_{N+1}$ is decreasing at $1/2$, then $%
M(p)\geq f_{N+1}(p)>f_{N+1}(1/2)$ for $p<1/2$ so $M(1/2)$ is minimal.

Finally, suppose $f_{N+1}$ is increasing at $1/2$. 
For $p<1/2$, $\log p < \log (1-p)$ and as $T_{N}=T_{N+1}$, we have $f_{N+1}(p)>f_{N}(p)$. 
Thus  $f_{N+1}$ is the unique maximum on some interval $(1/2-\varepsilon ,1/2) $
and we basically repeat the case 1a argument, beginning by finding the
largest $b_{1}<1/2$ where $f_{N+1}$ transitions to $f_{j}$.

\begin{remark}
The two examples shown in Figure \ref{fig:Jfunctions} illustrate two different ways that $M(p)$ can occur,
with $p$ being at an intersection (or transition) point in the top example and $p$ being at a local minimum
of some $f_j$ in the bottom example.

\end{remark}

\subsection{The special case of two IFS}

Suppose there are two IFSs, each of which has two children, $(a^{\alpha
_{i}},a^{\beta _{i}})$ for $i=1,2$, with $\beta _{1}-\alpha _{1}>\beta
_{2}-\alpha _{2}$, and chosen with equal likelihood. Let $\theta _{0}$ be
the unique solution to 
\begin{equation*}
\theta =\frac{2\log (1+a^{(\beta _{2}-\alpha _{2})\theta })}{(\alpha
_{1}+\alpha _{2})\left\vert \log a\right\vert }.
\end{equation*}

\begin{proposition}
(i) If $\beta _{1}-\alpha _{1}\geq 0>\beta _{2}-\alpha _{2}$, then the a.s. minimum 
measure dimension is
\begin{equation*}
    \min_{p}M(p)=f_{1}(1/2)=\frac{2\log 2}{(\alpha _{1}+\beta _{2})|\log a|)}.
\end{equation*}

(ii) If $\beta _{2}-\alpha _{2}\geq 0$, then $\min_{p}M(p)=\theta _{0}$ and
occurs with the choice 
\[
   p=p_{0}=1/(1+a^{(\beta _{2}-\alpha_{2})\theta _{0}}) > 1/2.
\]

(iii) If $\beta _{2}=\alpha _{2}$, then the almost sure minimal measure
dimension is 
\[  
   \min_p M(p) = \frac{2\log 2}{(\alpha _{1}+\alpha _{2})|\log a|}.
\]

\end{proposition}

\begin{proof}
(i) Here we have $N=1$ so \ $M(1/2)=f_{1}(1/2)$. Furthermore, as $S_{1}=1/2$%
, $f_{1}$ changes from decreasing to increasing at $1/2$. Thus $%
f_{1}(1/2)\leq f_{1}(p)$ for all $p$ so that $\min_{p}M(p)=f_{1}(1/2)$.

(ii) In this case, $N=2$ and $S_{2}=1$. Thus $f_{2}$ is decreasing on $(0,1)$
and $f_{1}$ increases on $[1/2,1)$. Hence Case 1b of the algorithm shows $%
\min_{p}M(p)=f_{2}(p^{\prime })$ where $p^{\prime }>1/2$ is the minimal $p$
where $f_{2}$ right transitions to $f_{1}$ at $p$. Of course we must have $%
f_{1}(p)=f_{2}(p)$ at the transition point $p$. But there can only be one
such choice of $p>1/2$ since $(f_{1}-f_{2})^{\prime }>0$ on $(1/2,1)$ and it
is a routine calculation to verify that $f_{1}(p_{0})=f_{2}(p_{0})$. Thus 
\begin{eqnarray*}
\min_{p}M(p) &=&M(p_{0})=f_{2}(p_{0})=\frac{2\log p_{0}}{(\alpha _{1}+\alpha
_{2})\log a} \\
&=&\frac{2\log (1+a^{(\beta _{2}-\alpha _{2})\theta _{0}})}{(\alpha
_{1}+\alpha _{2})\left\vert \log a\right\vert }=\theta _{0}.
\end{eqnarray*}

(iii) is immediate from (ii).
\end{proof}

\begin{remark}
We made the special assumption that $(\beta _{i}-\alpha _{i})$ was strictly
decreasing. This was essentially a notational convenience for the algorithm.
There is clearly no loss of generality in assuming the sequence is decreasing. 
But if, for example, we have $\beta _{j+1}-\alpha _{j+1}=\beta _{j}-\alpha _{j}$,
then the interval 
\begin{equation*}
\left( \frac{1}{1+a^{(\beta _{j+1}-\alpha _{j+1})\theta }},\frac{1}{%
1+a^{(\beta _{j}-\alpha _{j})\theta }}\right]
\end{equation*}%
is empty and consequently the function $f_{j}(p)$ is not present in the
definition of $G$. In Lemma \ref{transition} we saw that $M$ will transition
from $f_{i}$ to one of the functions $f_{j}$ with an adjacent domain. In the
strictly decreasing case these are the functions $f_{j}$ with $j=i\pm 1$.
But if, say, $f_{i+1}$ is not present, then the functions with domain
adjacent to that of $f_{i}$ will be $f_{i-1}$ and $f_{i+2}$ (assuming these
are both present). This notational modification is the key change one needs
to make to the algorithm when we do not have the strictly increasing
property.
\end{remark}

\subsection*{Acknowledgments}

The authors would like to thank an anonymous referee from the previous paper \cite{HM2} whose feedback first led them on the path to investigate the topics contained in this paper.

\end{document}